\documentclass[a4paper,oneside,11pt, reqno, psamsfonts]{amsart}
\usepackage{amsfonts,amssymb,amsmath,amsthm}
\usepackage{mathrsfs}
\usepackage{url}
\usepackage{enumerate}
\usepackage{xcolor}
\usepackage{todonotes}
\usepackage{ulem} 

\newtheorem{theorem}{Theorem}[section]
\newtheorem*{theorem*}{Theorem}
\newtheorem{lemma}[theorem]{Lemma}

\newtheorem{prop}[theorem]{Proposition}
\newtheorem{cor}[theorem]{Corollary}
\theoremstyle{definition}

\newtheorem{rem}[theorem]{Remark}
\newtheorem{example}[theorem]{Example}
\newtheorem{Question}[theorem]{Question}

\author[S. Ghara]{Soumitra Ghara}
\author[R. Gupta]{Rajeev Gupta}
\author[M. R. Reza]{Md. Ramiz Reza}

\address[S. Ghara]{Department of Mathematics and Statistics\\
Laval University\\
Quebec City - G1V 0A6}

\address[R. Gupta]{School of Mathematics and Computer Science\\
Indian Institute of Technology Goa\\
Goa -  403401}

\address[Md. R. Reza]{School of Mathematics\\
Indian Institute of Science Education and Research Thiruvananthapuram \\
Kerala - 695551}

\email[S. Ghara]{ghara90@gmail.com}
\email[R. Gupta]{rajeev.iisc@live.com}
\email[M. R. Reza]{ramiz.md@gmail.com}

\thanks{
The first author was supported by the Fields Post-doctoral Fellowship of the Fields Institute for Research in Mathematical Sciences, Canada and by a post-doctoral fellowship of Laval University, Canada. The second author was supported through the INSPIRE faculty grant (Ref. No. DST/INSPIRE/04/2017/002367).}

\begin{document}
\title [A local Douglas formula for higher order weighted Dirichlet-type integrals]{A local Douglas formula for higher order weighted Dirichlet-type integrals}
\begin{abstract}
We prove a local Douglas formula for  higher order weighted Dirichlet-type integrals. With the help of this formula, we study the multiplier algebra of the associated higher order weighted Dirichlet-type spaces $\mathcal H_{\pmb\mu},$ induced by an $m$-tuple
$\pmb \mu =(\mu_1,\ldots,\mu_{m})$ of finite non-negative Borel measures on the unit circle. In particular, it is shown that any weighted Dirichlet-type space of order $m,$ for $m\geqslant 3,$ forms an algebra under pointwise product. We also prove that every non-zero closed $M_z$-invariant subspace of $\mathcal H_{\pmb\mu},$ has codimension $1$ property if $m\geqslant 3$ or $\mu_2$ is finitely supported.
As another application of local Douglas formula obtained in this article, it is shown that for any $m\geqslant 2,$ weighted Dirichlet-type space of order $m$ does not coincide with any de Branges-Rovnyak space $\mathcal H(b)$ with equivalence of norms.
\end{abstract}


 \maketitle

\section{Introduction}
The symbols $\mathbb Z, \mathbb N$ and $\mathbb Z_{\geqslant 0}$ will denote the set of  integers, positive integers and non-negative integers respectively. We use the symbols $\mathbb T$ and $\mathbb D$ to denote the unit circle and the open unit disc in the complex plane $\mathbb C$ respectively. The notation $\mathcal M_{+}(\mathbb T)$ stands for the set of all finite non-negative Borel measures on $\mathbb T$. Let $\mathcal O(\mathbb D)$ denote the space of all complex valued holomorphic functions on $\mathbb D.$ For each $\alpha \in \mathbb R,$ let $\mathcal D_{\alpha}$ denote the subspace of $\mathcal O(\mathbb D)$ given by
\begin{equation*}
\mathcal D_{\alpha}:= \bigg\{f=\sum_{k=0}^\infty a_k z^k \Big|  \|f\|^2_{\alpha} := \sum_{k=0}^\infty(k+1)^{\alpha} |a_k|^2 < \infty \bigg\},
\end{equation*}
see \cite{TaylorDA} for a detailed study of the space $\mathcal D_{\alpha}.$ Many classical function spaces arise as  $\mathcal D_{\alpha}$ spaces. For examples $\mathcal D_{-1}$ coincides with the Bergman space, the space $\mathcal D_0$ coincides with the Hardy space $H^2$  and $\mathcal D_{1}$ coincides with the  Dirichlet space $\mathcal D.$ 
Recall that the Dirichlet space $\mathcal D$ consists of all  functions $f$ in $\mathcal O(\mathbb D)$ for which the  Dirichlet integral $D(f)$, given by
\begin{equation*}
D(f):= \displaystyle\int_{\mathbb D}\big|f^\prime(z)\big|^2 dA(z),
\end{equation*}
is finite. Here $A(z)$ denotes the {\it normalized Lebesgue area measure} on $\mathbb D.$ It turns out that $ \mathcal D$ is contained in the Hardy space $H^2.$  
In his study of minimal surfaces, Douglas \cite{Douglas} derived and used the following formula 
   \begin{align*}
   D(f)= \int_{\mathbb T}\int_{\mathbb T}\frac{|f^{*}(\zeta)- f^{*}(\lambda)|^2}{|\zeta-\lambda|^2}d\sigma(\zeta)d\sigma(\lambda),\,\,f\in \mathcal D,
   \end{align*}
where $\sigma(\zeta)$ denotes the {\it normalized arc length measure} on $\mathbb T$ and $f^{*}(\lambda)$ denotes the non-tangential limit of $f$ at $\lambda\in\mathbb T$ (which exists a.e. on $\mathbb T$). This and similar formulae turns out to be very useful in understanding the boundary behaviour of functions in $\mathcal D,$ see \cite{BeurlingA}. Later, Richter and Sundberg (see \cite[p. 356]{RS}) introduced the notion of the local Dirichlet-type integral $D_{\lambda}(f)$ of a function $f$ in $H^2$ at a point $\lambda\in\mathbb T$  in the following manner:
\begin{align*}
D_{\lambda}(f):= \int_{\mathbb T} \frac{|f^{*}(\zeta)- f^{*}(\lambda)|^2}{|\zeta-\lambda|^2}d\sigma(\zeta).
\end{align*}
 They established various remarkable formulae for the local Dirichlet-type integral in \cite[Sec. 3]{RS}. 
One of the motivations for studying the local Dirichlet-type integral was to understand the weighted Dirichlet-type integral $D_{\mu}(f)$ for a function $f\in\mathcal O(\mathbb D) $ and a measure $\mu\in \mathcal M_{+}(\mathbb T),$ a notion introduced and studied by Richter in order to describe all cyclic analytic $2$-isometries, see \cite[Theorem 5.1]{R}. 
The weighted Dirichlet-type integral $D_{\mu}(f)$ turns out to be a key ingredient in describing the structure of the invariant subspaces of the Dirichlet shift, that is, the operator $M_z$ of multiplication by the co-ordinate function $z$ on the Dirichlet space $\mathcal D,$ see \cite[Theorem 3.2]{RS-92}. For any $m\in\mathbb Z_{\geqslant 0},$ the operator $M_z$ on $\mathcal D_{m}$ turns out to be an analytic $(m+1)$ isometry, see \cite{AglerStan1}, \cite{AglerStan2}, \cite{AglerStan3} for a detailed study of $m$-isometries. 
Thus if one desires to describe the structure of the invariant subspaces of $M_z$ on $\mathcal D_{m},$  it is natural to look for model for analytic $(m+1)$ isometries, see \cite{Rydhe}. Following \cite{Rydhe}, for a measure $\mu$ in $\mathcal M_{+}(\mathbb T)$, $f\in \mathcal O(\mathbb D)$ and for a positive integer $n,$ we consider the {\it weighted Dirichlet-type integral of $f$ of order $n$} by 
\begin{align*}
D_{\mu,n}(f):= \frac{1}{n!(n-1)!}\displaystyle\int_{\mathbb D}\big|f^{(n)}(z)\big|^2 P_{\!\mu}(z)(1-|z|^2)^{n-1}dA(z).
\end{align*} 
Here $f^{(n)}(z)$ denotes the $n$-th order derivative of $f$ at $z,$ and $P_{\!\mu}(z)$ is the Poisson integral of the measure $\mu,$ that is,
\begin{align*}
P_{\!\mu}(z):=\int_{\mathbb T} \frac{1-|z|^2}{|z-\zeta|^2} d\mu(\zeta),\,\,\,z\in \mathbb D.
\end{align*}  
For a function $f\in \mathcal O(\mathbb D),$ it will be also useful for us to consider the {\it {weighted Dirichlet-type integral of order}} $0$  defined by
\begin{align*}
D_{\mu,0}(f):= \displaystyle\lim_{R\to 1^{-}}\displaystyle\int_{\mathbb T}\big|f(R\zeta)\big|^2 P_{\!\mu}(R\zeta) d\sigma(\zeta),
\end{align*}
provided the limit exists. 
Note that if $\mu=\sigma$ then the integral $D_{\sigma,0}(f)$ coincides with the square of the Hardy norm of  $f.$ In case of $n=1,$ the integral $D_{\mu,1}(f)$ coincides with the weighted Dirichlet-type integral $D_{\mu}(f)$ for every $\mu\in\mathcal M_+(\mathbb T).$ In particular if $\mu=\sigma$ then $D_{\sigma,1}(f)$ coincides with $D(f).$ If $\mu$ is $\delta_{\lambda}$, the Dirac measure at the point $\lambda \in \mathbb T$, we adopt a simpler notation $D_{\lambda,n}(\cdot)$ in place of $D_{\delta_{\lambda},n}(\cdot)$ and we refer it as a {\it local Dirichlet-type integral of order $n$ at $\lambda$}.   
We would like to mention that a similar notion, namely, local Dirichlet-type integral of order $n$ of a function at a point $\lambda\in\mathbb T$ is introduced and studied recently in \cite[p. 3]{LGR}. Nevertheless, we will see in  Remark \ref{Differs} that this notion differs from that of local Dirichlet-type integral of order $n$ introduced in this article.

For a measure $\mu\in \mathcal M_{+}(\mathbb T)/\{0\}$ and for each $n\in\mathbb Z_{\geqslant 0},$ we consider the semi-inner product space $\mathcal H_{\mu, n}$ given by
$$\mathcal H_{\mu, n}:=\big\{f\in \mathcal O(\mathbb D): D_{\mu,n}(f)< \infty\big\},$$
associated to the semi-norm $\sqrt{D_{\mu,n}(\cdot)}.$ In case $\mu$ is $\delta_{\lambda}$, $\lambda\in\mathbb T,$ we will use a simpler notation $\mathcal H_{\lambda,n}$ in place of $\mathcal H_{\delta_{\lambda},n}$ and we refer it as a {\it local Dirichlet space of order $n$ at $\lambda$}. If $\mu = 0,$ we set $\mathcal H_{\mu,n}= H^2$  for every $n\in \mathbb Z_{\geqslant 0}.$  If $\mu=\sigma,$ by a straightforward computation it follows that for a holomorphic function $f=\sum_{k=0}^\infty a_k z^k$ in $\mathcal O(\mathbb D),$ we have 
\begin{equation*}
D_{\sigma, n}(f)=\sum_{k=n}^\infty \binom{k}{n}|a_k|^2, \,\,\,\quad n\in\mathbb Z_{\geqslant 0},
\end{equation*}
where $\binom{k}{n}:=\frac{k!}{n!(k-n)!}$ for any $k\geqslant n.$ It follows that for each $n\in\mathbb Z_{\geqslant 0},$ the space $\mathcal H_{\sigma,n}$ coincides as a set with the space $\mathcal D_n.$ The reader is referred to \cite{GGR} (see also \cite{Rydhe}) for several properties of the spaces $\mathcal H_{\mu,n}$ for an arbitrary non-negative measure $\mu$ and a positive integer $n.$  

Richter and Sundberg had shown that, for any $\lambda\in\mathbb T,$ $D_{\lambda,1}(f)= D_{\lambda}(f),$ for every $f$ in $H^2,$ see \cite[Proposition 2.2]{RS}. In these notations, we note that
\begin{align*}
D_{\lambda,1}(f)= D_{\sigma,0} \Big(\frac{f(z)-f^{*}(\lambda)}{z-\lambda}\Big),
\end{align*}
for every $f\in \mathcal H_{\lambda,1},$ see also \cite[Proposition 1]{Sara} for an alternative proof. This formula is often known as a {\textit{local Douglas formula}}, see \cite[Section 7.2]{Primer}.
In this article, we establish a local Douglas formula for higher order local Dirichlet-type integrals, namely, we prove the following theorem.
\begin{theorem}\label{generalized local Douglas formula}
Let $n$ be a positive integer, $\lambda\in \mathbb T$, and $f\in\mathcal O (\mathbb D)$. Then $f\in \mathcal H_{\lambda, n}$ if and only if $f=\alpha+(z-\lambda)g$ for some $g$ in $\mathcal H_{\sigma,n-1}$ and $\alpha\in \mathbb C.$ 
Moreover, in this case, the following statements hold:\begin{itemize}
\item[(i)]$D_{\lambda,n}(f) =D_{\sigma, n-1}(g),$
\item[(ii)]$f(z)\to \alpha$ as $z\to\lambda$ in each oricyclic approach region $|z-\lambda|<\kappa (1-|z|^2)^{\frac{1}{2}}$, $\kappa>0$. In particular, $f^*(\lambda)$ exists and is equal to $\alpha.$
\end{itemize} 
\end{theorem}
Thus one can rewrite the local Douglas formula for higher order Dirichlet-type integrals in the following form:
\begin{align}\label{GRS formula}
D_{\lambda,n}(f)= D_{\sigma,n-1} \Big(\frac{f(z)-f^{*}(\lambda)}{z-\lambda}\Big),\,\,\,\,f\in\mathcal H_{\lambda,n},\,\,\,n\in\mathbb N. 
\end{align}
We explore various applications of this generalized local Douglas formula \eqref{GRS formula} in this article. For $0<r<1,$ let $f_r$ be the $r$-dilation of a function $f\in \mathcal O(\mathbb D)$ given by $f_r(z):=f(rz),\,z\in\mathbb D.$ As a first application of the formula \eqref{GRS formula}, we prove in Theorem \ref{C_r contraction} that for each $n\geqslant 2$ and $\mu\in \mathcal M_+(\mathbb T),$ 
\begin{equation}\label{stronger ineq for dilation}
D_{\mu,n}(f_r)\leqslant  \frac{4^{n-1}(2-r)r^{2n}}{(1+r)^{2n-2}} D_{\mu,n}(f)
\end{equation}
for every $f\in \mathcal H_{\mu,n}$ and $0<r<1$. 
Note that  for any $n\in\mathbb N$ and $0<r<1$, we have $\frac{4^{n-1}(2-r)r^{2n}}{(1+r)^{2n-2}}\leqslant 1$. Thus the inequality \eqref{stronger ineq for dilation} is an improvement of \cite[Theorem 1.1]{GGR}, where only the contractivity of the map $f\mapsto f_r$ was shown. Analogous results are known in the case of $n=1,$ see \cite{RS}, \cite{Sara}, \cite{Primer}, \cite{Rangroup}, \cite{MasRans} and references therein.

In the study of function spaces, it is an important problem to determine the associated multiplier algebra. Let $\mu\in \mathcal M_+(\mathbb T)$ and $n$ be a non-negative integer.  The multiplier algebra of $\mathcal H_{\mu,n}$ is defined by 
\begin{align*}
\text{Mult\,}(\mathcal H_{\mu,n}) := \{\varphi\in \mathcal O({\mathbb D}) : \varphi f \in \mathcal H_{\mu,n}\,\text{\,\,for every\,} \,\,f\in \mathcal H_{\mu,n}\}.
\end{align*} 
Several studies have been done in characterizing the multipliers of the  Dirichlet space $\mathcal D$, and more generally for the multipliers of the weighted Dirichlet-type space $\mathcal H_{\mu,1}$ for any $\mu\in \mathcal M_+(\mathbb T),$  see \cite{Stegenga},\cite{RS-92},\cite{Chartrand} and references therein. In Theorem \ref{multipliers in local space}, we characterize $\text{Mult\,}(\mathcal H_{\lambda,n})$ for every $\lambda\in\mathbb T$ and $n\in\mathbb N.$ This characterization also helps us to study the $\text{Mult\,}(\mathcal H_{\mu,n})$ for any $\mu\in \mathcal M_+(\mathbb T)$ and for every $n\in\mathbb N.$  It is well known that the space $\mathcal D_n$ forms an algebra under pointwise product for every $n\geqslant 2$ and consequently $\text{Mult\,}(\mathcal D_n) = \mathcal D_n$ for each $n\geqslant 2,$  see \cite[Remark, p. 239]{TaylorDA}.
Even though $\mathcal D_1,$ the classical Dirichlet space,  is not an algebra under pointwise product (see \cite[Theorem 1.3.1]{Primer}), the subspace of bounded functions in it turns out to be an algebra, see \cite[Theorem 1.3.2]{Primer}. More generally, it is well known that $\mathcal H_{\mu,1}$ is never an algebra under pointwise product for any $\mu\in \mathcal M_+(\mathbb T),$  see \cite[Sec. 7.1, Exercise 4]{Primer}, but $\mathcal H_{\mu,1}\cap H^{\infty}$ forms an algebra under pointwise product, see \cite[p. 170]{RS-92}, where $H^\infty$ is the Banach algebra of bounded holomorphic functions on $\mathbb D$. We extend this result in the following theorem for higher order weighted Dirichlet-type spaces. 
\begin{theorem}\label{prop algebra}
Let $\mu\in \mathcal M_+(\mathbb T)$ and $n\in \mathbb N$. Then $\mathcal H_{\mu,n}\cap \text{Mult\,}(\mathcal H_{\sigma,n-1})$ is an algebra.
\end{theorem}
As a consequence of this theorem, we find that $\mathcal H_{\mu,n}$ also forms an algebra under pointwise product and $\text{Mult\,}(\mathcal H_{\mu,n}) = \mathcal H_{\mu,n}$ for every $\mu\in \mathcal M_+(\mathbb T)/\{0\}$ and  for each $n\geqslant 3,$ see Corollary \ref{Algebra}. The answer to the question of whether $\mathcal H_{\mu,2}$ is an algebra or not, is of mixed nature. We have found that when $\mu=\delta_{\lambda}$ with $\lambda\in\mathbb T,$ then $\mathcal H_{\lambda,2}$ is not an algebra under pointwise product. On the other hand we have  $\mathcal H_{\sigma,2}=\mathcal D_2$ forms an algebra under pointwise product, see \cite[Remark, p. 239]{TaylorDA}.

Let $m\in\mathbb N$ and $\pmb \mu= (\mu_1,\ldots,\mu_{m})$ be an $m$-tuple of non-negative measures in $\mathcal M_{+}(\mathbb T).$  Let  $\mathcal H_{\pmb \mu}$ be the linear subspace of $\mathcal O(\mathbb D)$ given by $\mathcal H_{\pmb \mu}:= \bigcap_{j=1}^{m} \mathcal H_{\mu_j,j}.$
Note that for each $j=1,\ldots,m,$ the space $\mathcal H_{\mu_j,n}$ is contained in the Hardy space $H^2$ for any $n\in\mathbb N,$ see \cite[Corollary 2.5]{GGR}. For $f\in \mathcal H_{\pmb \mu},$ we associate the norm $\|f\|_{\pmb \mu}$ given by
\begin{align*}
\|f\|^2_{\pmb \mu}:= \|f\|^2_{\!_{H^2}}+ \sum\limits_{j=1}^{m} D_{\mu_j,j}(f).
\end{align*} 
The space $\mathcal H_{\pmb \mu}$ with respect to the norm $\|\cdot\|_{\pmb\mu}$ turns out to be a reproducing kernel Hilbert space, see \cite[p. 13]{GGR}. 
In particular, it follows that $\mathcal H_{\mu,n}$ is a Hilbert space with respect to the norm $\|\cdot\|_{\mu,n},$ given by
\begin{align*}
\|f\|_{\mu,n}^2 = \|f\|^2_{H^2}+ D_{\mu,n}(f),\,\,\,\,f\in \mathcal H_{\mu,n}.
\end{align*}
The operator $M_z$ of multiplication by the co-ordinate function on $\mathcal H_{\pmb\mu}$ turns out to be a cyclic analytic $(m+1)$-isometry, see \cite[Theorem 4.1, Corollary 5.3]{GGR}. 
Let $Lat(M_z,\mathcal H_{\pmb\mu})$ denote the lattice of all closed $M_z$-invariant subspaces of $\mathcal H_{\pmb\mu}.$ A subspace $\mathcal W$ in $Lat(M_z,\mathcal H_{\pmb\mu})$ is said to have codimension $k$ property if $ dim (\mathcal W\ominus (z-\lambda)\mathcal W)=k,$ for each $\lambda\in \mathbb D,$ see \cite[Definition 2.12]{Rinvbanach}. 
In order to have a Beurling-type description of $Lat(M_z,\mathcal H_{\pmb\mu}),$ it is expected that $\mathcal W$ has codimension $1$ property for every non-zero $\mathcal W$ in $Lat(M_z,\mathcal H_{\pmb\mu}),$ see \cite[Proposition 2.13]{Rinvbanach}. 
In the case of $m=1,$ that is,  when $\pmb\mu=\mu_1$ and $\mu_1\in \mathcal M_{+}(\mathbb T),$ 
it is well known that  any non-zero $\mathcal W$ in $Lat(M_z,\mathcal H_{\pmb\mu})$ has codimension $1$ property, see \cite[Theorem 3.2]{RS-92}. 
We find that for any $m\geqslant 3$ and any non-zero $\mathcal W$ in $Lat(M_z,\mathcal H_{\pmb\mu})$ has codimension $1$ property. For the remaining case $m=2,$ that is, when $\pmb\mu= (\mu_1,\mu_2),$ we show that the result remains true provided $\mu_2$ is finitely atomic, see Theorem \ref{case m=2}.


There is an intimate connection between the local Dirichlet spaces $\mathcal H_{\lambda}$ and the de Branges-Rovnyak spaces $\mathcal H(b)$ (see \cite{deBR} for basic theory of $\mathcal H(b)$ spaces). For any $b$ in the unit ball of $H^\infty $, the de Branges-Rovnyak space $\mathcal H(b)$ is the reproducing kernel Hilbert space with the reproducing kernel $\frac{1-b(z)\bar{b(w)}}{1-z\bar{w}}$, $z,w\in\mathbb D$.
In \cite{Sara}, Sarason showed that any local Dirichlet space $\mathcal H_{\lambda}$ coincides with a de Branges-Ronyak space $\mathcal H(b),$ with equality of norms, where $b$ can be chosen to be $b(z)=\frac{(1-r)\overline{\lambda}z}{1-r\overline{\lambda}z},$ $z\in\mathbb D,$ with $r=\frac{3-\sqrt{5}}{2}.$ 
In \cite[Theorem 3.1]{CGR-2010}, Chevrot, Guillot and Ransford showed that, in fact, if any $\mathcal H_\mu$ space is equal to some de Branges-Ronyak space $\mathcal H(b)$, with equality of norms, then $\mu$ has to be a positive multiple of a point mass measure for some $\lambda\in\mathbb T$ and further $b$ is identified explicitly. Thereafter, in \cite{CostaraRansford}, the problem of characterizing $\mu$ and $b$ such that $\mathcal H_\mu=\mathcal H(b)$ as a set was addressed
 by obtaining certain necessary and some sufficient conditions for $\mathcal H_\mu=\mathcal H(b).$ In this paper, we study the same question for the spaces $\mathcal H_{\pmb \mu}$ and $\mathcal H(b),$ 
where $\pmb\mu=(\mu_1,\ldots,\mu_{m})$, $m\geqslant 2$, is an arbitrary $m$-tuple of positive measures in $\mathcal M_{+}(\mathbb T)$. 

This article is organized as follows. Section \ref{Local Douglas Formula} is devoted to prove Theorem \ref{generalized local Douglas formula}. Section \ref{Multipliers} deals with a few applications of generalized local Douglas Formula. In this section, we prove Theorem \ref{prop algebra} and as a corollary it is observed that $\mathcal H_{\mu,n}$ is an algebra for any $\mu\in\mathcal M_+(\mathbb T)/\{0\}$ and $n\geqslant 3,$ see Corollary \ref{Algebra}.  
In Section \ref{Cod}, we discuss about the codimension $1$ property for any closed $M_z$-invariant subspaces of $\mathcal H_{\mu,n}.$ In Section \ref{Relationship with H(b)}, it is shown that, for $n\geqslant 2,$ any weighted Dirichlet-type space of order $n$ does not coincide with de Branges-Rovnyak space $H(b)$, with equivalence of norms, for any $b$ in the unit ball of $H^\infty$.


\section{A Generalized Local Douglas Formula}\label{Local Douglas Formula}
In this section, we shall prove a local Douglas formula for local Dirichlet spaces of order $n$.  For any $\lambda\in \mathbb T$ and $n \in \mathbb N$, let $\nu_{\lambda, n}$ denote the positive weighted area measure on $\mathbb D$ given by
 \begin{eqnarray*}
 d\nu_{\lambda, n}(z)=\frac{1}{n! (n-1)!} P_{\delta_{\lambda}}(z)(1-|z|^2)^{n-1}dA(z),
 \end{eqnarray*}
where $P_{\delta_{\lambda}}(z)= \frac{1-|z|^2}{|z-\lambda|^2},$ $z\in\mathbb D,$ is the {\it Poisson integral} of the Dirac measure $\delta_{\lambda}.$ Let $A^2(d\nu_{\lambda, n})$ be the weighted Bergman space
corresponding to the measure $\nu_{\lambda, n}$, that is,
$$A^2(d\nu_{\lambda, n}):=\{f\in \mathcal O(\mathbb D): \|f\|^2=\int_{\mathbb D}|f(z)|^2d\nu_{\lambda, n}(z)< \infty\}.$$

\begin{lemma}\label{lem kernel of A_2(nu)}
For any $\lambda\in \mathbb T$ and $n\in \mathbb N$, $A^2(d\nu_{\lambda, n})$ is a reproducing kernel Hilbert space with the reproducing kernel  $$(n+1)!(n-1)!  \frac{(z-\lambda)(\overline{w}-\overline{\lambda})}{(1-z\overline{w})^{n+2}},~z ,w\in \mathbb D.$$
\end{lemma}
\begin{proof}
Let $A^2\big((1-|z|^2)^n dA(z)\big)$ be the weighted Bergman space corresponding to the weighted area measure $(1-|z|^2)^n dA(z)$ on $\mathbb D,$ that is,
\begin{align*}
A^2\big((1-|z|^2)^n dA(z)\big):=\big\{f\in \mathcal O(\mathbb D): \|f\|^2=\int_{\mathbb D}|f(z)|^2(1-|z|^2)^n dA(z)< \infty\big\}.
\end{align*}
It is evident that the operator
$\sqrt{n! (n-1)!}M_{z-\lambda}$ is a unitary from $A^2\big((1-|z|^2)^n dA(z)\big)$ onto $A^2\big(d\nu_{\lambda, n}\big)$, where $M_{z-\lambda}$ denotes the operator of multiplication by the function $(z-\lambda)$. By a direct computation, it is verified that 
$$\left\{\sqrt{(n+1){n+k+1 \choose k}}z^k : k\in \mathbb Z_{\geqslant 0}\right\}$$
forms an orthonormal basis of $A^2\big((1-|z|^2)^n dA(z)\big)$. Therefore, by \cite[Theorem 2.4]{PAULRKHS}, $A^2(d\nu_{\lambda, n})$ has the reproducing kernel 
$$n!(n-1)!(n+1) (z-\lambda)(\overline{w}-\overline{\lambda})\sum_{k=0}^\infty \binom{n+k+1}{k} z^k\overline{w}^k=(n+1)!(n-1)! \frac{(z-\lambda)(\overline{w}-\overline{\lambda})}{(1-z\overline{w})^{n+2}}, \quad z,w\in \mathbb D.$$
This completes the proof.
\end{proof}

\begin{lemma}\label{lem well defined norm}
Let $\lambda\in \mathbb T$, $n\in \mathbb N$, and $f\in \mathcal H_{\sigma, n-1}.$ Then $((z-\lambda)f)^{(n)}(z)=0$ for all $z\in \mathbb D$ if and only if $D_{\sigma, n-1}(f)=0$.
\end{lemma}
\begin{proof} Suppose that $f\in \mathcal H_{\sigma, n-1}$ and $((z-\lambda)f)^{(n)}(z)=0$ for all $z\in \mathbb D.$  Then $(z-\lambda)f$ is a polynomial of degree at most $n-1$. Note that $\mathcal H_{\sigma, n-1}$ is contained in the Hardy space $H^2$(see \cite[Corollary 2.5]{GGR}, \cite[Corollary 1, Theorem 7]{TaylorDA}). Therefore, in view of \cite[Lemma 1.5.4]{Primer}, we find that  
$\lim_{|z|\to 1^{-}}(1-|z|^2)|f(z)|^2=0$. Hence it follows that
$\lim_{r\to 1^{-}}(r\lambda-\lambda)f(r\lambda)=0.$ This gives us that the polynomial $(z-\lambda)f$ must have a zero at $\lambda$ and consequently $f$ is a polynomial of degree  at most $n-2.$ Therefore, $D_{\sigma, n-1}(f)=0.$
 Conversely, assume that $D_{\sigma, n-1}(f)=0$. This implies $f$ is a polynomial of degree at most $n-2$, and therefore $((z-\lambda)f)^{(n)}(z)= 0$ for all $z\in \mathbb D$. 
\end{proof}

Let $j\in\mathbb Z_{\geqslant 0}.$ Note that, in general $\sqrt{D_{\sigma,j}(\cdot)}$ represents a seminorm on $\mathcal H_{\sigma,j}$ (unless $j=0$). Consider, the subspace $\mathcal U_{\sigma,j}$ of $\mathcal H_{\sigma,j}$ given by 
$$\mathcal U_{\sigma,j}= \big\{f \in \mathcal O(\mathbb D): f(z)=\sum_{k=j}^{\infty}a_kz^k, D_{\sigma,j}(f) < \infty \big\}.$$
It is straightforward to verify that  $\sqrt{D_{\sigma,j}(\cdot)}$ represents a norm on $\mathcal U_{\sigma,j}$ which is induced by an inner product. Moreover $\mathcal U_{\sigma,j}$ is a reproducing kernel Hilbert space  in which the monomials given by $\{ {\binom{k}{j}}^{-\frac{1}{2}} z^k :k\geqslant j\}$ forms an orthonormal basis. The reproducing kernel $K_{\sigma,j}$ of $\mathcal U_{\sigma, j}$ is given by
\begin{equation}\label{kernel of standard weighted Dirichlet space}
K_{\sigma, j}(z,w)=\sum_{k=j}^\infty \binom{k}{j}^{-1} z^k\overline{w}^k, \quad z, w\in \mathbb D.
\end{equation}

For any $\alpha>0$ and $j\in\mathbb Z_{\geqslant 0}$, the Pochammer symbol $(\alpha)_{j}$  is defined by 
\begin{equation*}
(\alpha)_j=\begin{cases}\alpha(\alpha+1)\cdots(\alpha+j-1),~&  j\geqslant 1\\
1,~& j=0. \end{cases}
\end{equation*}
\begin{prop}\label{isometry of T}
Let $n$ be a positive integer, $\lambda\in \mathbb T$, and $T:\mathcal O(\mathbb D)\to \mathcal O(\mathbb D)$ be the linear map given by
$$(Tf)(z)=((z-\lambda)f)^{(n)}(z),~z\in \mathbb D,~f\in \mathcal O(\mathbb D).$$
Then $T$ maps the Hilbert space $(\mathcal U_{\sigma, n-1}, \sqrt{D_{\sigma, n-1}(\cdot)})$ isometrically onto $A^2(d\nu_{\lambda, n}).$
\end{prop}
\begin{proof}
Consider the space $\mathcal W:= T(\mathcal U_{\sigma, n-1})$ endowed with the  norm 
$$\|T(f)\|_{\!_\mathcal W}:=\sqrt{D_{\sigma, n-1}(f)},\quad f\in \mathcal H_{\sigma,n-1}.$$
By Lemma \ref{lem well defined norm}, $\|\cdot\|_{\!_\mathcal W}$ is  a well-defined norm on $\mathcal W$. Clearly $T$ maps $\mathcal U_{\sigma, n-1}$ isometrically onto $\mathcal W$. This gives us that the space $\mathcal W$ endowed with the  norm $\|\cdot\|_{\!_\mathcal W}$ is a Hilbert space. Now we show that $\mathcal W$ is a reproducing kernel Hilbert space. This is equivalent to showing that for each $w\in\mathbb D,$ the evaluation map $E_w: \mathcal W \to \mathbb C,$ defined by $E_w(g)=g(w),$  is continuous. 
Let $w\in\mathbb D$, and $T(f_n)$ converges to $T(f)$ in $\mathcal W$, where $f_n, f\in \mathcal U_{\sigma, n-1}$. Clearly $f_n$ converges to $f$ in $\mathcal U_{\sigma, n-1}$. Since $\mathcal U_{\sigma,n-1}$
is a reproducing kernel Hilbert space, it follows that 
$f_n$ converges to $f$ uniformly on compact subsets of $\mathbb D$. Therefore $T(f_n)$ also converges to $T(f)$  uniformly on compact subsets of $\mathbb D$. This, in particular, shows that $E_w$ is continuous on $\mathcal W$, proving that $\mathcal W$ is a reproducing kernel Hilbert space.

The proposition will now be established by showing 
that the reproducing kernels of $\mathcal W$ and $A^2(d\nu)$ coincide (see \cite[Proposition 2.3]{PAULRKHS}).
As the operator $T$ is unitary, it maps the orthonormal basis $\big\{\binom{k}{n-1}^{-\frac{1}{2}}z^k:k\geqslant n-1\big\}$ of  $\mathcal U_{\sigma,n-1}$ to an orthonormal basis of $\mathcal W.$ 
Thus the reproducing kernel of $\mathcal W$ is given by  
\begin{equation}\label{kernel of K}
K_{\!_\mathcal W}(z,w)=\frac{\partial^{2n}}{\partial z^n \partial\overline{w}^n}\Big((z-\lambda)(\overline{w}-\overline{\lambda})K_{\sigma,n-1}(z,w)\Big),\quad z,w\in \mathbb D.
\end{equation}
In view of \eqref{kernel of standard weighted Dirichlet space}, we see that the expression in \eqref{kernel of K} is equal to
\begin{align*}
&\frac{\partial^{2n}}{\partial z^n \partial\overline{w}^n}\left(\sum_{k=n-1}^\infty{k\choose n-1}^{-1}\Big(z^{k+1}\overline{w}^{k+1}+z^{k}\overline{w}^{k}-\lambda z^{k}\overline{w}^{k+1}-\overline{\lambda}  z^{k+1}\overline{w}^{k}\Big)\right)\\
=&\sum_{k=n-1}^\infty \frac{1}{{k\choose n-1}} \big((k+2-n)_n\big)^2 \, z^{k+1-n}\overline{w}^{k+1-n}+\sum_{k=n}^\infty \frac{1}{{k\choose n-1}}\big((k+1-n)_n\big)^2  z^{k-n}\overline{w}^{k-n}\\
&\hspace{1.2cm}-\sum_{k=n}^\infty \frac{1}{{k\choose n-1}}(k+2-n)_n(k+1-n)_n\Big(\lambda z^{k-n}\overline{w}^{k+1-n}+\overline{\lambda} z^{k+1-n}\overline{w}^{k-n}\Big)\\
=&\sum_{k=0}^\infty\big((k+1)_n\big)^2\bigg({k+n-1\choose n-1}^{-1}+{k+n\choose n-1}^{-1}\bigg)z^{k}\overline{w}^{k}\\&\hspace{2.7cm}-\sum_{k=0}^\infty {k+n\choose n-1}^{-1}(k+2)_n(k+1)_n\Big(\lambda z^{k}\overline{w}^{k+1}+\overline{\lambda}z^{k+1}\overline{w}^{k}\Big)\\
=&(n-1)!\Bigg(\sum_{k=0}^\infty (k+1)_n(2k+n+1) z^{k}\overline{w}^{k}-\sum_{k=0}^\infty (k+1)_{n+1}\Big(\lambda z^{k}\overline{w}^{k+1}+\overline{\lambda}z^{k+1}\overline{w}^{k}\Big)\Bigg).
\end{align*}
Also, by a straightforward calculation, we see that 
\begin{align*} 
&~~~\frac{(z-\lambda)(\overline{w}-\overline{\lambda})}{(1-z\overline{w})^{n+2}}\\
&=\sum_{k=0}^\infty {n+k+1\choose k}\Big(z^{k+1}\overline{w}^{k+1} +  z^{k}\overline{w}^{k}-\lambda z^{k}\overline{w}^{k+1}-\overline{\lambda} z^{k+1}\overline{w}^{k}\Big)\\
&=1+\sum_{k=1}^\infty\Bigg({n+k\choose k-1}+{n+k+1\choose k}\Bigg)z^k\overline{w}^k
-\sum_{k=0}^\infty {n+k+1\choose k} \Big(\lambda z^{k}\overline{w}^{k+1}+\overline{\lambda}z^{k+1}\overline{w}^{k}\Big)\\
&=1+\frac{1}{(n+1)!}\bigg(\sum_{k=1}^\infty (k+1)_n(2k+n+1)z^k\overline{w}^k-\sum_{k=0}^\infty (k+1)_{n+1}\Big(\lambda z^{k}\overline{w}^{k+1}+\overline{\lambda}z^{k+1}\overline{w}^{k}\Big)\bigg)\\
&=\frac{1}{(n+1)!}\bigg(\sum_{k=0}^\infty (k+1)_n(2k+n+1)z^k\overline{w}^k-\sum_{k=0}^\infty (k+1)_{n+1}\Big(\lambda z^{k}\overline{w}^{k+1}+\overline{\lambda}z^{k+1}\overline{w}^{k}\Big)\bigg)\\
&= \frac{1}{(n+1)!(n-1)!}K_{\!_\mathcal W}(z,w).
\end{align*}
Use of Lemma \ref{lem kernel of A_2(nu)} now completes the proof.
\end{proof}

For any holomorphic function $f$ on the unit disc $\mathbb D$ and a complex number $\lambda$ in  $\mathbb T$, let $f^*(\lambda)$ denote the boundary value $f$ at $\lambda$, defined by 
$$f^*(\lambda):=\lim_{r\to 1^-}f(r\lambda),$$
whenever it exists.  
The following theorem (main theorem of this section) is deduced as a consequence of Lemma \ref{isometry of T}. This theorem, in particular, tells us that for any $f$ in $\mathcal H_{\lambda,n},$ the boundary value of $f$  at $\lambda$ exists. As an another consequence of Theorem \ref{generalized local Douglas formula} and  \cite[Proposition 2.6]{GGR}, we see that $\mathcal H_{\lambda,n}$ is contained in $\mathcal H_{\sigma,n-1}.$


\noindent \textbf{Proof of Theorem \ref{generalized local Douglas formula}:}
Assume that $f\in \mathcal H_{\lambda, n}$. Thus we have  $f^{(n)}\in A^2(d\nu_{\lambda,n})$. Now, by Proposition 
\ref{isometry of T}, there exists a $\tilde{g}$ in $\mathcal U_{\sigma,n-1}$ such that $D_{\lambda,n}(f)=D_{\sigma, n-1}(\tilde{g})$ and
$((z-\lambda)\tilde{g})^{(n)}=f^{(n)}$. 
Therefore,
\begin{eqnarray*}
f(z)=(z-\lambda)\tilde{g}(z)+p(z), \quad z\in\mathbb D,
\end{eqnarray*}
for some polynomial $p$ of degree at most $n-1$. Thus, with $q(z)=\frac{p(z)-p(\lambda)}{z-\lambda},$ we have 
\begin{eqnarray*}
f(z)=p(\lambda)+(z-\lambda)g(z), z\in \mathbb D
\end{eqnarray*}  
where $g=\tilde{g}+q$.
Clearly, $D_{\sigma, n-1}(g)=D_{\sigma, n-1}(\tilde{g})=D_{\lambda,n}(f),$ which is finite. Hence $g\in \mathcal H_{\sigma, n-1}$. Conversely, assume that  $f=\alpha+(z-\lambda)g$ for some $g$ 
with $g\in \mathcal H_{\sigma,n-1}.$ Then it follows from Lemma \ref{isometry of T} that 
$f^{(n)}\in A^2(d\nu_{\lambda,n})$, and hence $f\in \mathcal H_{\lambda, n}$. This completes the proof of the first of the theorem. 
Note that, for $f\in \mathcal H_{\lambda,n}$, the equality $D_{\lambda,n}(f) =D_{\sigma, n-1}(g)$ is already established. To complete the proof, suppose that $f=\alpha+(z-\lambda)g$ for some complex number $\alpha$ and $g$ in $\mathcal H_{\sigma,n-1}$. 
Since $\mathcal H_{\sigma,n-1}$ is contained in the Hardy space $H^2$, by  \cite[Lemma 1.5.4]{Primer} we have 
$\lim_{|z|\to 1}(1-|z|^2)^{\frac{1}{2}}|g(z)|=0$.
If $|z-\lambda|<\kappa (1-|z|^2)^{\frac{1}{2}}$, $\kappa>0$,  then 
$$|f(z)-\alpha|<\kappa (1-|z|^2)^{\frac{1}{2}}|g(z)|.$$ Hence it follows that $f(z)\to \alpha$ as $z\to\lambda$ in each oricyclic approach region $|z-\lambda|<\kappa (1-|z|^2)^{\frac{1}{2}}$. This completes the proof of this Theorem.

%
\begin{cor}
Let $n$ be a positive integer and $\mu$ be a  finite positive Borel measure on the unit circle $\mathbb T$. Then for any  $f$ in $\mathcal H_{\mu,n}$,  $f^*(\lambda)$ exists for $\mu$-almost every $\lambda\in \mathbb T$, and the following holds: 
\begin{equation*}
D_{\mu,n}(f)=\int_{\mathbb T}\int_{\mathbb D} \Big|\Big(\frac{f(z)-f^*(\lambda)}{z-\lambda}\Big)^{(n-1)}(z)\Big |^2(1-|z|^2)^{n-1} dA(z) d\mu(\lambda)
\end{equation*}
\end{cor}
\begin{proof}
By  Fubini's theorem, 
\begin{align*}
D_{\mu,n}(f)=\frac{1}{n!(n-1)!}\displaystyle\int_{\mathbb D}\big|f^{(n)}(z)\big|^2 P_{\!\mu}(z)(1-|z|^2)^{n-1}dA(z)
=\int_{\mathbb \mathbb T} D_{\lambda,n}(f) d\mu(\lambda).
\end{align*}
An application of Theorem \ref{generalized local Douglas formula} completes the proof of  the corollary. 
\end{proof}

We finish this section with a few applications of Theorem 
\ref{generalized local Douglas formula}. The first application shows that for any $n> 1$ the map $f\mapsto f_r$, $0\leqslant r <1$, is contractive on $\mathcal H_{\mu,n}$. Here for any $f\in \mathcal O(\mathbb D)$, the $r$th dilation $f_r$ of $f$ is defined by $f_r(z):=f(rz)$, $z\in \mathbb{D}.$ We point out that this result is already obtained in \cite[Proposition 3]{Sara} for $n=1$, and \cite[Theorem 1.1]{GGR} for every $n\geqslant 1.$  One of the consequences of this result is that $\lim_{r\to 1^-} D_{\mu,n}(f_r-f)=0$ for any $f\in \mathcal H_{\mu,n}.$ We start with the following lemma.
\begin{lemma}\label{lemma f_r}
Let $\mathcal H$ be a semi-inner product space of holomorphic functions on the unit disc $\mathbb D$ with the semi-norm $\|\cdot\|$. Suppose that $zf\in \mathcal H$ whenever $f\in \mathcal H$. If $\|zf\| \geqslant \|f\|$ for all $f\in \mathcal H$ then for $\lambda \in \mathbb T$, $0\leqslant r<1$, we have
\begin{equation*}
\|(z-\lambda)f\|\leqslant \frac{2}{1+r}\|(z-r\lambda)f\|,~f\in \mathcal H.
\end{equation*}
\end{lemma}
\begin{proof} Let $f\in \mathcal H$. Since $\|zf\|\geqslant \|f\|$, it follows that $$|\mbox{Re}~\lambda\langle f, zf\rangle |\leqslant |\lambda\langle f, zf\rangle |\leqslant\|f\| \|zf\|\leqslant \|zf\|^2,$$
where for any complex number $w$, $\mbox{Re}\,w$ denotes the real part of $w$. Hence  
\begin{align*}
&\frac{4}{(1+r)^2}\|(z-r\lambda)f\|^2-\|(z-\lambda)f)\|^2\\&= \Big(\frac{4}{(1+r)^2}-1\Big)\|zf\|^2-\Big(1-\frac{4r^2}{(1+r)^2}\Big)\|f\|^2+2 \Big(1-\frac{4r}{(1+r)^2}\Big)~ \mbox{Re}~\lambda\langle f, zf\rangle\\
&\geqslant \Big(\big(\frac{4}{(1+r)^2}-1\big)  -\big(1-\frac{4r^2}{(1+r)^2}\big)- 2\big(1-\frac{4r}{(1+r)^2}\big)\Big)\|zf\|^2\\
&=0.
\end{align*}
This completes the proof.
%
\end{proof}

\begin{theorem}\label{C_r contraction}
Let $n\geqslant 1$ be an integer, and $\mu$ be in $\mathcal M_+(\mathbb T).$ Then for any  $f$ in $\mathcal O(\mathbb D)$, we have 
\begin{equation}\label{Contractivity of C_r}
D_{\mu,n}(f_r)\leqslant  \frac{4^{n-1}(2-r)r^{2n}}{(1+r)^{2n-2}} D_{\mu,n}(f),~ ~0\leqslant r <1.
\end{equation}
\end{theorem}
\begin{proof}
We will establish the proof by induction on $n.$ The base case of $n=1$ follows from \cite[Corollary 1.5]{MasRans}. 
Assume that the statement holds for every $f\in\mathcal O(\mathbb D),$ $0\leqslant r<1,$ and $n=1,\ldots,k-1$ for $k\geqslant 2.$  Let  $f\in \mathcal O(\mathbb D)$ and $0\leqslant r<1.$
Since $D_{\mu,k}(f)=\int_{\mathbb T}D_{\lambda,k}(f)d\mu(\lambda)$, in order to establish \eqref{Contractivity of C_r} for $n=k$, it suffices to show that the  \eqref{Contractivity of C_r} holds for $n=k$ and $\mu=\delta_{\lambda}$  for each $\lambda\in \mathbb T.$
Fix a $\lambda$ in $\mathbb T$ and assume that  
$D_{\lambda,k}(f)<\infty.$ Then by Theorem \ref{generalized local Douglas formula} we obtain
that $f(z)=f^*(\lambda)+(z-\lambda)g(z)$ for some $g$ in $\mathcal H_{\sigma, k-1}$ and $D_{\lambda,k}(f)=D_{\sigma,k-1}(g)$. By a straightforward computation, we see that 
\begin{equation}\label{eqncontr2}
f_r(z)=f_r^*(\lambda)+(z-\lambda)h_r(z),
\end{equation}
where
\begin{align*}
h(z)&=r\frac{(z-\lambda)g(z)-(r\lambda-\lambda)g(r\lambda)}{z-r\lambda}.
\end{align*}
Note that 
$$h(z)=r \frac{(z-\lambda)(g(z)-g(r\lambda))+(z-r\lambda)g(r\lambda)}{z-r\lambda}.$$
Using this together with Lemma \ref{lemma f_r} we get that
\begin{align}\label{eqhg}
\notag D_{\sigma,k-1}(h)&=r^2D_{\sigma,k-1}\left(\frac{(z-\lambda)(g(z)-g(r\lambda))}{z-r\lambda}\right)\\ \notag
&\leqslant \frac{4r^2}{(1+r)^2} D_{\sigma,k-1}(g(z)-g(r\lambda))\\
&=\frac{4r^2}{(1+r)^2} D_{\sigma,k-1}(g).
\end{align}
By induction hypothesis, we have that 
\begin{equation}\label{eqncontra1}
D_{\sigma,k-1}(h_r)\leqslant  \frac{4^{k-2}(2-r)r^{2k-2}}{(1+r)^{2k-4}} D_{\sigma,k-1}(h).
\end{equation}
Using  Theorem \ref{generalized local Douglas formula},  together with \eqref{eqncontr2}, \eqref{eqhg} and \eqref{eqncontra1} we get
\begin{align*}
D_{\lambda,k}(f_r) &=D_{\sigma,k-1}(h_r) \\
&\leqslant \frac{4^{k-2}(2-r)r^{2k-2}}{(1+r)^{2k-4}} D_{\sigma,k-1}(h)\\
&\leqslant \frac{4^{k-1}(2-r)r^{2k}}{(1+r)^{2k-2}} D_{\sigma,k-1}(g)=\frac{4^{k-1}(2-r)r^{2k}}{(1+r)^{2k-2}} D_{\lambda,k}(f).
\end{align*}
This completes the proof.

\end{proof}
%

The following theorem which is a generalization of  \cite[Theorem 3.1]{CFS} (see also \cite[Lemma 4.2]{CostaraRansford}) describes the weighted Dirichlet-type space $\mathcal H_{\mu, n}$ when $\mu$ is finitely supported.
\begin{theorem}
Let $n, k$ be two positive integers, and $\mu=\sum_{j=1}^k c_j \delta_{\lambda_j}$ for some  $c_j>0$ and $\lambda_j\in \mathbb T$, $j=1,\ldots,k$. Then
$f\in \mathcal H_{\mu, n}$ if and only if there exists a polynomial $p$ of degree at most $k-1$ and a function $g\in \mathcal H_{\sigma, n-1}$ such that
\begin{equation*}
f=p+ g\prod_{j=1}^k (z-\lambda_j) .
\end{equation*}
\end{theorem}
\begin{proof}
Note that $f\in  \mathcal H_{\mu, n}$ if and only if $f\in  \mathcal H_{\lambda_j, n}$ for each $j=1,\ldots,k.$ Suppose that $f=p+ g\prod_{j=1}^k (z-\lambda_j) $ for some polynomial
$p$ of degree at most $k-1$ and a function $g\in \mathcal H_{\sigma, n-1}$. Since $z$ is a multiplier of $\mathcal H_{\sigma,n-1}$, it follows that for any $i=1,\ldots,k$, $g \prod_{j=1\\j\neq i}^k(z-\lambda_j) $ is in
$\mathcal H_{\sigma,n-1}$. Therefore, by Theorem \ref{generalized local Douglas formula}, the function
$g \prod_{j=1}^k (z-\lambda_j) \in \mathcal H_{\lambda_i, n}$ for each $i=1,\ldots,k$. Hence $f\in \mathcal H_{\mu,n}.$\\
Conversely, suppose that $f\in \mathcal H_{\mu,n}$ and assume that, without loss of generality, $\lambda_1,\ldots,\lambda_k$ are distinct.
Then $f\in \mathcal H_{\lambda_j, n},$ and applying Theorem \ref{generalized local Douglas formula} we see that $f^*(\lambda_j)$ exists for all $j=1,\ldots,k$. Let $p$ be the polynomial of degree $k-1$ such that $p(\lambda_j)=f^*(\lambda_j)$,  $j=1,\ldots,k$. Again, by Theorem  \ref{generalized local Douglas formula}, it follows that $\frac{f-p}{z-\lambda_j}$ belongs to $\mathcal H_{\sigma, n-1}$ for $j=1,\ldots, k.$ Since $\lambda_1,\ldots,\lambda_k$ are distinct, using partial fraction formula,  we see that
\begin{equation}\label{eqn partial fraction}
\frac{f-p}{\prod_{j=1}^k (z-\lambda_j)}=\sum_{j=1}^k \alpha_j \frac{f-p}{z-\lambda_j}
\end{equation}
for some $\alpha_1,\ldots,\alpha_k \in \mathbb C.$
As the right hand side in \eqref{eqn partial fraction} belongs to $\mathcal H_{\sigma,n-1}$, this completes the proof.
\end{proof}


\section{Multipliers of $\mathcal H_{\mu,n}$ }\label{Multipliers}
In the study of any function space, it is an important problem to characterize the multipliers of the associated function space. Let $\mu$ be a non-zero measure in $\mathcal M_+(\mathbb T)$ and $n$ be a positive integer.  The set of multipliers of $\mathcal H_{\mu,n},$ is denoted by $\text{Mult\,}(\mathcal H_{\mu,n})$ and is defined by 
\begin{align*}
\text{Mult\,}(\mathcal H_{\mu,n}) = \{\varphi\in \mathcal O(\mathbb D) : \varphi f \in \mathcal H_{\mu,n}\,\text{\,\,for every\,} \,\,f\in \mathcal H_{\mu,n}\}.
\end{align*}
In this section, we deduce many properties of a multiplier of $\mathcal H_{\mu,n}$ and give some descriptions of the multiplier space $\text{Mult\,}(\mathcal H_{\mu,n})$ in several different ways. 
Since the co-ordinate function $z$ is a multiplier of $\mathcal H_{\mu,n}$ (see  \cite[Proposition 2.6]{GGR}), it follows that the set of all polynomials is contained in $\text{Mult\,}(\mathcal H_{\mu,n}).$ Let $\mathcal O(\overline{\mathbb D})$ denote the set of functions which are holomorphic in some neighbourhood of the closed unit disc $\overline{\mathbb D}$. As an application of the following proposition we show that if $\varphi \in \mathcal O(\overline{\mathbb D})$ then $\varphi\in \text{Mult\,}(\mathcal H_{\mu,n}).$ Before proving the proposition, we start with the following useful lemma.

\begin{lemma}\label{lower derivative inclusion}
Let $\mu\in \mathcal M_+(\mathbb T)\setminus \{0\}$ and $n\in \mathbb N$. Then for any $f\in \mathcal H_{\mu,n}$ and $j=0,\ldots,n,$ the integral 
\begin{align*}
I_{\mu,j,n}(f):=\int_{\mathbb D}|f^{(j)}(z)|^2P_{\mu}(z)(1-|z|^2)^{n-1}dA(z)
\end{align*}
is finite.
\end{lemma}
\begin{proof}

 Let $A^2(d\rho)$ be the weighted Bergman space with the weighted area measure $\rho$ on the unit disc given by $d\rho(z)=P_{\mu}(z)(1-|z|^2)^{n-1}dA(z)$. 
Note that,  a function $f$ in $\mathcal O(\mathbb D)$ satisfies
$I_{\mu,j,n}(f) < \infty$ if and only if $f^{(j)}\in A^2(d\rho)$. Let $f\in \mathcal H_{\mu,n}$. By definition it follows that $f^{(n)}\in  A^2(d\rho)$. Let $0\leqslant k\leqslant n-1$. The proof will be established by showing that if $f^{(j)}\in A^2(d\rho)$ for all  $j=k+1,\ldots, n$, then $f^{(k)}\in A^2(d\rho)$. Note that, by the Leibniz's product rule for differentiation
\begin{equation}\label{leib}
(z^{n-k}f)^{(n)}=z^{n-k}f^{(n)}+c_0z^{n-k-1}f^{(n-1)}+\cdots + c_{n-k}zf^{(k+1)}+ {n\choose n-k}(n-k)!f^{(k)}
\end{equation}  
for some constants $c_0,\ldots,c_{n-k}$. Since $z$ is a multiplier of $\mathcal H_{\mu, n}$ (see \cite[Proposition 2.6]{GGR}), we see that 
 $(z^{n-k}f)^{(n)}\in A^2(d\rho)$.  Since $f^{(j)}\in A^2(d\rho)$ for all  $j=k+1,\ldots, n$, it follows from 
 \eqref{leib} that $f^{(k)}\in A^2(d\rho).$

\end{proof}

\begin{prop}\label{Useful lemma}
Let $\mu\in \mathcal M_+(\mathbb T)$ and $n\in \mathbb N$. Suppose $ \varphi \in \mathcal O (\mathbb D)$ such that $\varphi^{(n)} \in H^{\infty}.$ Then $\varphi\in \text{Mult\,}(\mathcal H_{\mu,n}).$
\end{prop}
\begin{proof}
Let $\varphi\in\mathcal O(\mathbb D)$ satisfying $\varphi^{(n)} \in H^{\infty}.$ By repeated applications of the Mean Value Theorem, it follows that $\varphi^{(j)} \in H^{\infty}$ for each $j=0,1,\ldots,n.$ By the Leibniz's product rule for differentiation and the Cauchy-Schwarz's inequality, note that, for any $f\in \mathcal H_{\mu,n},$  
\begin{align*}
|(\varphi f)^{(n)}(z)|^2 \leqslant \Big(\sum\limits_{j=0}^n |\varphi^{(n-j)}(z)|^2 |f^{(j)}(z)|^2\Big) \Big(\sum\limits_{j=0}^n {\binom{n}{j}}^2\Big),\,\,\,z\in\mathbb D.
\end{align*}
Let $C_n=\sum_{j=0}^n {\binom{n}{j}}^2$ and $\|g\|_{\infty}:=\sup\{|g(z)|:z\in\mathbb D\}$ for any $g\in H^{\infty}.$ It follows that for any $f\in \mathcal H_{\mu,n},$  
\begin{align*}
D_{\mu,n}(\varphi f) \leqslant \frac{C_n}{n!(n-1)!} \sum\limits_{j=0}^n \|\varphi^{(n-j)}\|^2_{\infty}\int_{\mathbb D}|f^{(j)}(z)|^2P_{\mu}(z)(1-|z|^2)^{n-1}dA(z).
\end{align*}
Applying Lemma \ref{lower derivative inclusion}, we obtain that $D_{\mu,n}(\varphi f) < \infty.$ Hence it follows that $\varphi \in \text{Mult\,}(\mathcal H_{\mu,n}).$  
\end{proof}
\begin{cor}
Let $\mu\in \mathcal M_+(\mathbb T)$ and $n\in \mathbb N$. Then $\mathcal O(\overline{\mathbb D}) \subset \text{Mult\,}(\mathcal H_{\mu,n}).$
\end{cor}
\begin{proof}
Note that for any $\varphi \in  \mathcal O(\overline{\mathbb D}),$ we have $\varphi^{(j)}\in H^{\infty}$ for every $j\in\mathbb N.$ An application of  Proposition \ref{Useful lemma} now completes the proof.
\end{proof}

Let $\pmb{\mu}=(\mu_1,\ldots,\mu_m),$ where $\mu_j\in\mathcal M_+(\mathbb T)$ for every $j=1,\ldots,m.$ Recall that the linear space $\mathcal H_{\pmb{\mu}} = \cap_{j=1}^m \mathcal H_{\mu_j,j}$ is a Hilbert space with respect to the norm $\|\cdot\|_{\pmb{\mu}},$ given by 
\begin{align*}
\|f\|_{\pmb{\mu}}^2 = \|f\|^2_{H^2}+\sum\limits_{j=1}^m D_{\mu_j,j}(f),\quad f\in \mathcal H_{\pmb{\mu}}.
\end{align*}   
Furthermore, the evaluation functional $ev_z,$ defined by $ev_z(f):= f(z),\,\,f\in \mathcal H_{\pmb{\mu}},$ is bounded for every $z\in\mathbb D.$ Thus $\mathcal H_{\pmb{\mu}}$ is a reproducing kernel Hilbert space of holomorphic functions on the unit disc,  see \cite{ARO}, \cite{PAULRKHS} for definition and other basic properties of a  reproducing kernel Hilbert space. In particular, it follows that  for any  finite positive Borel measure $\mu$ on $\mathbb T$ and a positive integer $n,$ the linear space $\mathcal H_{\mu,n}$ is a reproducing kernel Hilbert space with respect to the norm $\|\cdot\|_{\mu,n},$ given by
\begin{align*}
\|f\|_{\mu,n}^2 = \|f\|^2_{H^2}+ D_{\mu,n}(f),\,\,\,\,f\in \mathcal H_{\mu,n}.
\end{align*}
Many properties of  $\text{Mult\,}(\mathcal H_{\mu,n})$ can be derived using the reproducing kernel Hilbert space structure of $\mathcal H_{\mu,n}.$  From now onwards we will often use the notation $(\mathcal H_{\mu,n},\|\cdot\|_{\mu,n})$ to denote the Hilbert space $\mathcal H_{\mu,n}$ equipped with the norm $\|\cdot\|_{\mu,n}.$

In the following part of this section, we consider a general reproducing kernel Hilbert space having some special properties and derive few properties of the associated multiplier space.   

Let $\mathscr H_\kappa$ be a reproducing kernel Hilbert space of $\mathbb C$-valued holomorphic functions defined on the unit disc $\mathbb D$ and let $\kappa : \mathbb D \times \mathbb D \rightarrow \mathbb C$ be the sesqui-analytic reproducing kernel for $\mathscr H_\kappa,$ that is, $\kappa(\cdot, w) \in \mathscr H_\kappa$ and 
\begin{align*}
\langle f, \kappa(\cdot, w)\rangle_{\mathscr H_\kappa} = f(w), \quad f \in \mathscr H_\kappa, ~w \in \mathbb D.
\end{align*}
 
In the remaining part of this section, we shall assume the following properties on $\kappa$:
\begin{enumerate}
\item[(A1)] $\kappa$ is normalized at the origin, that is, $\kappa(w, 0)=1$ for every $w \in \mathbb D,$
\item[(A2)] $z$ is a multiplier for $\mathscr H_\kappa,$ that is, $zf \in \mathscr H_\kappa$ for every $f \in \mathscr H_\kappa,$
\item[(A3)] $\mathbb C[z],$ the set of all polynomials, is dense in $\mathscr H_\kappa.$
\end{enumerate} 
We refer to the reproducing kernel Hilbert space $\mathscr H_{\kappa}$ satisfying (A1)-(A3) as the{\it{  functional Hilbert space with property (A)}}.  

For example take $\pmb{\mu}=(\mu_1,\ldots,\mu_m),$ where $\mu_j\in\mathcal M_+(\mathbb T)$ for every $j=1,\ldots,m$ and consider the associated reproducing kernel Hilbert space $\mathcal H_{\pmb{\mu}}.$ Since $\langle f, 1\rangle= f(0)$ for every $f\in \mathcal H_{\pmb{\mu}},$ the kernel function associated to $(\mathcal H_{\pmb{\mu}},\|\cdot\|_{\pmb{\mu}})$ is normalized at the origin. This together with \cite[Theorem 4.1, Proposition 4.2]{GGR} gives us that $(\mathcal H_{\pmb{\mu}},\|\cdot\|_{\pmb{\mu}})$ is a functional Hilbert space with property (A). In particular, it follows that, for any finite positive Borel measure $\mu$ on $ \mathbb T$ and for any $n\in\mathbb N,$ the Hilbert space $(\mathcal H_{\mu,n},\|\cdot\|_{\mu,n})$ is also a functional Hilbert space with property (A). 

For a reproducing kernel Hilbert space $\mathscr H_{\kappa},$ which consists of holomorphic functions on the unit disc $\mathbb D,$ the multipliers space of $\mathscr H_{\kappa},$ is defined by 
\begin{align*}
\text{Mult\,}(\mathscr H_{\kappa}) := \{\varphi\in \mathscr H_{\kappa}: \varphi f \in \mathscr H_{\kappa}\,\text{\,\,for every\,} \,\,f\in\mathscr H_{\kappa}\}.
\end{align*}
Note that  $\text{Mult\,} (\mathscr H_{\kappa})$ forms an algebra with respect to the pointwise product, that is,  $\varphi \psi \in \text{Mult\,}(\mathscr H_{\kappa})$ for every $\varphi,\psi \in\text{Mult\,} (\mathscr H_{\kappa}).$ Moreover one can identify $\text{Mult\,} (\mathscr H_{\kappa})$  as a closed subalgebra of $\mathscr B(\mathscr H_{\kappa})$ in the following manner. Given $\varphi\in  \text{Mult\,}(\mathscr H_{\kappa}),$ we write $M_{\varphi}$ for the multiplication operator acting on $\mathscr H_{\kappa}$ given by
\begin{align*}
M_{\varphi}(f)=\varphi f,\,\,f\in \mathscr H_{\kappa}.
\end{align*} 
By an application of the Closed Graph Theorem, it follows that for any $\varphi\in \text{Mult\,}(\mathscr H_{\kappa}),$ the operator $M_{\varphi}$ is  bounded from $\mathscr H_{\kappa}$ into itself. Consider the norm $\|\cdot\|_{op}$ on $\text{Mult\,}(\mathscr H_{\kappa})$ defined by 
\begin{align*}
\|\varphi\|_{op}:= \|M_{\varphi}\| = \sup \{ \|\varphi f \|_{\mathscr H_{\kappa}} : \|f\|_{\mathscr H_{\kappa}}=1\},\,\,\,\varphi\in \text{Mult\,}(\mathscr H_{\kappa}).
\end{align*} 
It follows that $\big(\text{Mult\,}(\mathscr H_{\kappa}),\|\cdot\|_{op}\big)$ is a commutative Banach algebra with identity.  
 Let  $\mathscr H_{\kappa}$  be a functional Hilbert space with property (A). It follows that $\text{Mult\,}(\mathscr H_{\kappa})$ can be identified as the commutant of the operator $M_z$ of multiplication by the co-ordinate function on the Hilbert space $\mathscr H_{\kappa}.$ Indeed if $T\in \mathscr B(\mathscr H_{\kappa})$ satisfying $ TM_z= M_zT,$ then a routine argument using the polynomial density of the functional Hilbert space $\mathscr H_{\kappa},$ it can be shown that $T$ is given by $T=M_{\varphi}$ for some $\varphi\in \text{Mult\,}(\mathscr H_{\kappa})$ where $\varphi= T(1)$.

Let  $\mathscr H_{\kappa}$  be a functional Hilbert space with property (A). By assumption (A1), we have $\langle \kappa(\cdot,0), \kappa (\cdot,w)\rangle= \kappa(w,0)=1$ for each $w\in\mathbb D.$ Therefore $\kappa(\cdot,w) \neq 0$ for each $w\in \mathbb D.$ Hence it follows that the multiplier algebra $\text{Mult\,}(\mathscr H_{\kappa})$ is a subalgebra of $H^{\infty},$ the algebra of bounded analytic functions on the unit disc, for every functional Hilbert space $\mathscr H_{\kappa}$ with property (A), see \cite[Problem 68]{HALP}, \cite[Corollary 5.22]{PAULRKHS}, \cite[Proposition 3]{SH}. Thus, we obtain that $\text{Mult\,}(\mathcal H_{\mu,n}) \subseteq H^{\infty}$ for every measure $\mu$ in $\mathcal M_+(\mathbb T)$ and $n\in \mathbb N$.

Among the measures in $\mathcal M_+(\mathbb T),$ there are two special measures. First one is the Lebesgue measure $\sigma$ and the second one is the Dirac delta measure $\delta_{\lambda}$ for $\lambda\in\mathbb T.$ 
The space $\mathcal H_{\sigma,n}$ coincides with the weighted Dirichlet-type space $\mathcal D_n$ for each $n\in\mathbb Z_{\geqslant 0}.$ The set of multipliers of $\mathcal D_n$ is well studied, see for example \cite{TaylorDA, Stegenga} and references therein. 
It is known that, for every $n\geqslant 2,$ the space $\mathcal D_n$ forms an algebra, that is, $fg\in \mathcal D_n,$ for every $f,g\in\mathcal D_n,$ see \cite[Theorem 7]{TaylorDA}. Consequently it follows that $\text{Mult\,}(\mathcal H_{\sigma,n}) = \mathcal H_{\sigma,n}$ for every $n\geqslant 2.$  Now we concentrate our study to characterize the multiplier of $\mathcal H_{\lambda,n}$ for any positive integer $n.$ The following lemma will be useful in describing the set $\text{Mult\,}(\mathcal H_{\lambda,n}).$ This lemma can be thought as a generalization of \cite[Lemma 5.3]{RS}.

\begin{lemma}\label{Multiplier property}
Let $n$ be a positive integer, $\lambda \in \mathbb T$  and $f\in\mathcal H_{\lambda,n}$ such that $D_{\lambda,n}(f) \neq 0.$ Suppose $\varphi \in \text{Mult\,}(\mathcal H_{\sigma,n-1}) \cap \mathcal H_{\lambda,n} ,$ then 
\begin{align*}
D_{\lambda,n}(\varphi f) & \leqslant 2 \|\varphi\|_{s (\sigma,n-1)}^2 D_{\lambda,n}(f) + 2 |f^*(\lambda)|^2D_{\lambda,n}(\varphi)\,\,\,\,\text{and}\\
 |f^*(\lambda)|^2D_{\lambda,n}(\varphi) & \leqslant 2 \|\varphi\|_{s (\sigma,n-1)}^2 D_{\lambda,n}(f) + 2 D_{\lambda,n}(\varphi f),
\end{align*}
where 
$$\| \varphi\|_{s (\sigma,n-1)}:= \sup\Big\{ \sqrt{D_{\sigma,n-1}(\varphi h)} : D_{\sigma,n-1}(h) = 1\Big\}.$$ 
\end{lemma}
\begin{proof}
For a function $g\in H^2,$ if the non-tangential limit $g^*(\lambda)$ of $g$ at $\lambda$ exists, then by the generalized local Douglas formula (see Theorem \ref{generalized local Douglas formula}), we have
\begin{align*}
D_{\lambda,n}(g) = D_{\sigma, n-1}\Big(\frac{g(z)-g(\lambda)}{z-\lambda}\Big).
\end{align*}
Since $f,\varphi \in\mathcal H_{\lambda,n},$ the non-tangential limits $f^*(\lambda)$ and $\varphi^*(\lambda),$ of $f$ and $\varphi$ at $\lambda$  respectively, exist, see Theorem \ref{generalized local Douglas formula}. It follows that 
\begin{align*}
D_{\lambda,n}(\varphi f) = D_{\sigma, n-1}\Big(\frac{\varphi f(z)-(\varphi f)^*(\lambda)}{z-\lambda}\Big).
\end{align*} 
Using the identity 
 \begin{align}\label{useful id}
\frac{\varphi f(z)-\varphi f(\lambda)}{z-\lambda} = \varphi(z) \frac{f(z)- f^*(\lambda)}{z-\lambda} + f^*(\lambda) \frac{\varphi(z)-\varphi (\lambda)}{z-\lambda},\,\,\,\,z\in\mathbb D,
\end{align}
it follows that 
\begin{align*}
D_{\lambda,n}(\varphi f) &\leqslant 2 D_{\sigma,n-1}\Big(\varphi \frac{f(z)-f^*(\lambda)}{z-\lambda}\Big) + 2|f^*(\lambda)|^2 D_{\sigma,n-1}\Big( \frac{\varphi(z)-\varphi(\lambda)}{z-\lambda}\Big),\\
& \leqslant  2  \|\varphi\|_{s (\sigma,n-1)}^2 D_{\sigma,n-1}\Big( \frac{f(z)-f^*(\lambda)}{z-\lambda}\Big) + 2 |f^*(\lambda)|^2D_{\sigma,n-1}\Big( \frac{\varphi(z)-\varphi^*(\lambda)}{z-\lambda}\Big),\\
&\leqslant  2 \|\varphi\|_{s (\sigma,n-1)}^2 D_{\lambda,n}(f) + 2 |f^*(\lambda)|^2D_{\lambda,n}(\varphi) .
\end{align*}
Using the identity \eqref{useful id} one more time, we obtain that 
\begin{align*}
|f^*(\lambda)|^2D_{\lambda,n}(\varphi)
 &= D_{\sigma,n-1}\Big(f^*(\lambda)\frac{\varphi(z)-\varphi(\lambda)}{z-\lambda}\Big),\\
 &= D_{\sigma,n-1} \Big( \frac{\varphi f(z)-\varphi f(\lambda)}{z-\lambda} - \varphi(z) \frac{f(z)- f^*(\lambda)}{z-\lambda} \Big),\\
 &\leqslant 2 D_{\sigma,n-1} \Big( \frac{\varphi f(z)-(\varphi f)^*(\lambda)}{z-\lambda} \Big) + 2D_{\sigma,n-1} \Big( \varphi(z) \frac{f(z)- f^*(\lambda)}{z-\lambda} \Big),\\
 &\leqslant 2 D_{\lambda,n}(\varphi f) + 2 \|\varphi\|_{s (\sigma,n-1)}^2 D_{\sigma,n-1} \Big( \frac{f(z)- f^*(\lambda)}{z-\lambda} \Big),\\
 &= 2 D_{\lambda,n}(\varphi f) + 2 \|\varphi\|_{s (\sigma,n-1)}^2 D_{\lambda,n}(f).
\end{align*}
This completes the proof.
\end{proof}

The following proposition describes the multiplier algebra of $\mathcal H_{\lambda, n}$ for any $n\in\mathbb N$ and $\lambda\in\mathbb T.$ 
\begin{theorem}\label{multipliers in local space}
Let $n\in\mathbb N$ and $\lambda \in \mathbb T.$ Then, $\text{Mult\,}(\mathcal H_{\lambda,n})= \mathcal H_{\lambda,n}\cap  \text{Mult\,}(\mathcal H_{\sigma,n-1}).$
\end{theorem}
\begin{proof}
By Lemma \ref{Multiplier property}, it follows that $\varphi \in \text{Mult\,}(\mathcal H_{\lambda,n})$ whenever $\varphi \in \mathcal H_{\lambda,n}\cap  \text{Mult\,}(\mathcal H_{\sigma,n-1}).$ 
For the converse, let $\varphi \in \text{Mult\,}(\mathcal H_{\lambda,n}).$ Since the constant function $1\in \mathcal H_{\lambda,n},$ it follows that $\varphi \in \mathcal H_{\lambda,n}.$ Consider the map $L_{\lambda} : \mathcal H_{\lambda,n} \rightarrow \mathcal H_{\sigma,n-1}$  defined by 
\begin{align*}
L_{\lambda}(f)(z):= \frac{f(z)-f^*(\lambda)}{z-\lambda},\,\,\,z\in\mathbb D.
\end{align*}
By the generalized local Douglas formula in Theorem \ref{generalized local Douglas formula}, we have $L_{\lambda}$ is an isometry w.r.t the associated semi-norms, that is, $D_{\sigma,n-1}(L_{\lambda}f)= D_{\lambda,n}(f)$ for every $f\in  \mathcal H_{\lambda,n}.$ Furthermore it also follows that $L_{\lambda}$ is onto. 
Let $f\in \mathcal H_{\lambda,n}.$ An application of  the identity \eqref{useful id} gives us 
\begin{align*}
D_{\sigma,n-1}\Big(\varphi(z)\frac{f(z)-f^*(\lambda)}{z-\lambda}\Big)&= D_{\sigma,n-1} \Big( \frac{\varphi f(z)-(\varphi f)^*(\lambda)}{z-\lambda} - f^*(\lambda) \frac{\varphi(z)- \varphi^*(\lambda)}{z-\lambda} \Big),\\
&\leqslant 2 D_{\sigma,n-1} \Big( \frac{\varphi f(z)-(\varphi f)^*(\lambda)}{z-\lambda}\Big) + 2 |f^*(\lambda)|^2 D_{\sigma,n-1}\Big( \frac{\varphi(z)- \varphi^*(\lambda)}{z-\lambda} \Big),\\
&= 2 D_{\lambda,n}(\varphi f) + 2|f^*(\lambda)|^2 D_{\lambda,n}(\varphi),
\end{align*}
Thus we have $D_{\sigma,n-1}(\varphi L_{\lambda}f) < \infty$ for every $f\in \mathcal H_{\lambda,n}.$ Since $L_{\lambda}$ is onto, we obtain that $D_{\sigma,n-1}(\varphi g) < \infty$ for every $g\in \mathcal H_{\sigma,n-1}.$ Hence $\varphi \in \text{Mult\,}(\mathcal H_{\sigma,n-1}).$ This completes the proof.
\end{proof}



\vspace{2mm}
%
\noindent \textbf{Proof of Theorem \ref{prop algebra} :}
Let $\varphi, \psi\in \mathcal H_{\mu,n}\cap \text{Mult\,}(\mathcal H_{\sigma,n-1})$. Since $ \text{Mult\,}(\mathcal H_{\sigma,n-1})$ is an algebra, it is sufficient to show that $\varphi\psi \in \mathcal H_{\mu,n}.$ Note that $D_{\mu,n}(g) = \int_{\mathbb T} D_{\lambda,n}(g) d\mu(\lambda),$ for any $g\in\mathcal H_{\mu,n}.$ Hence we obtain  that $\varphi, \psi\in \mathcal H_{\lambda,n}$ for $[\mu]$ a.e. $\lambda$ in $\mathbb T.$ 
By Lemma \ref{Multiplier property} we have
\begin{align*}
D_{\lambda,n}(\varphi \psi)&\leqslant 2\|\varphi\|_{s (\sigma,n-1)}^2D_{\lambda,n}(\psi)+2|\psi^*(\lambda)|^2D_{\lambda,n}(\varphi),\\
&\leqslant 2\|\varphi\|_{s (\sigma,n-1)}^2D_{\lambda,n}(\psi)+2\|\psi\|_{\infty}^2D_{\lambda,n}(\varphi),~\lambda\in \mathbb T/A.
\end{align*}
Here in the last line we used the fact that $\psi\in \text{Mult\,}(\mathcal H_{\sigma,n-1})\subseteq H^\infty$. 
Integrating both sides  w.r.t $\mu(\lambda)$ we obtain
\begin{equation*}
D_{\mu,n}(\varphi \psi)\leqslant 2\|\varphi\|_{s (\sigma,n-1)}^2D_{\mu,n}(\psi)+2\|\psi\|_{\infty}^2D_{\mu,n}(\varphi),~\lambda\in \mathbb T.
\end{equation*}
This completes the proof of the theorem.

The above theorem helps us to characterize the multipliers of $\mathcal H_{\mu,n}$ for any $\mu\in \mathcal M_+(\mathbb T)$ and for any positive integer $n\geqslant 3.$ In this case, it turns out that $\mathcal H_{\mu,n}$ forms an algebra and consequently the multipliers of $\mathcal H_{\mu,n}$ is itself.

\begin{cor}\label{Algebra}
Let $n\geqslant 3$ be a positive integer and $\mu\in \mathcal M_+(\mathbb T)$. Then $\mathcal H_{\mu,n}$ is an algebra and consequently $\text{Mult\,}(\mathcal H_{\mu, n})=\mathcal H_{\mu,n}.$ 
\end{cor}
\begin{proof}
We have $\text{Mult\,}(\mathcal H_{\sigma, n-1})=\mathcal H_{\sigma,n-1}$ for any $n\geqslant 3.$ Thus applying Theorem \ref{Algebra}, it follows that $\mathcal H_{\mu,n}\cap \mathcal H_{\sigma,n-1}$ is an algebra. 
By \cite[Lemma 2.4]{GGR}, we also have that $\mathcal H_{\mu,n}\subseteq \mathcal H_{\sigma,n-1}.$ Hence $\mathcal H_{\mu,n}\cap \mathcal H_{\sigma,n-1}=\mathcal H_{\mu,n}.$ Thus $\mathcal H_{\mu,n}$ is an algebra and consequently $\text{Mult\,}(\mathcal H_{\mu, n})=\mathcal H_{\mu,n}.$ 
\end{proof}

\begin{cor}\label{algebra for tuple}
Let $\pmb \mu =(\mu_1,\ldots,\mu_{m})$ be an $m$-tuple of finite positive Borel measures on $\mathbb T$ with $\mu_m \neq 0.$ Then $\mathcal H_{\pmb\mu}\cap \text{Mult\,}(\mathcal H_{\sigma,m-1})$ is an algebra. In particular, $\mathcal H_{\pmb\mu}$ is an algebra for $m\geqslant 3.$ 
\end{cor}
\begin{proof}
It is well known that for every $n\in \mathbb N,$ $\text{Mult\,}(\mathcal H_{\sigma,n})\subset \text{Mult\,}(\mathcal H_{\sigma,n-1}),$ see \cite[p. 233, Corollary 1]{TaylorDA}. So, we have 
\begin{eqnarray*}
\mathcal H_{\pmb\mu}\cap \text{Mult\,}(\mathcal H_{\sigma,m-1}) &=&  \mathcal H_{\mu_1,1}\cap \cdots \cap \mathcal H_{\mu_m,m} \cap \text{Mult\,}(\mathcal H_{\sigma,m-1}),\\
&=& \big(\mathcal H_{\mu_1,1}\cap \text{Mult\,}(\mathcal H_{\sigma,0})\big)\cap  \cdots \cap \big(\mathcal H_{\mu_m,m} \cap \text{Mult\,}(\mathcal H_{\sigma,m-1})\big).
\end{eqnarray*}
By an application of Theorem \ref{Algebra}, $\mathcal H_{\mu_j,j} \cap \text{Mult\,}(\mathcal H_{\sigma,j-1})$ is an algebra for each $j=1,\ldots,m.$ Hence $\mathcal H_{\pmb\mu}\cap \text{Mult\,}(\mathcal H_{\sigma,m-1})$ is an algebra. 
It also follows from Theorem \ref{Algebra} that the space $\mathcal H_{\mu,j} \cap \text{Mult\,}(\mathcal H_{\sigma,j-1})=\mathcal H_{\mu,j}$
for any $j\geqslant 3.$ Hence $\mathcal H_{\pmb\mu}\cap \text{Mult\,}(\mathcal H_{\sigma,m-1})=\mathcal H_{\pmb\mu}$ for any $m\geqslant 3.$ This completes the proof.
\end{proof}

Although $\mathcal H_{\mu,n}$ is an algebra for every $n\geqslant 3$ and for any $\mu \in \mathcal M_+(\mathbb T)$, it is well known that $\mathcal H_{\mu,1}$ is never an algebra, see \cite[Exercise 7.1.4]{Primer}. So it is natural to ask the following question.
\begin{Question}Characterize all $\mu\in \mathcal M_+(\mathbb T)$ such that $\mathcal H_{\mu,2}$ is an algebra?
\end{Question}


\begin{example}\label{remalgebra} 
We now point out few cases when $\mathcal H_{\mu, 2}$ is an algebra and when it is not.
\begin{itemize}
\item[(\rm i)] The space $\mathcal H_{\lambda,2}$ is not an algebra for any $\lambda\in\mathbb T$. If possible assume that  $\mathcal H_{\lambda,2}$ is  an algebra for some $\lambda\in\mathbb T$. Then we must have $
\mathcal H_{\lambda,2}=\text{Mult\,}(\mathcal H_{\lambda,2})\subset H^\infty$. We now show that there exists a function in $\mathcal H_{\lambda,2}$ which is not bounded on $\mathbb D$. Choose a $\tilde{\lambda}\neq \lambda\in \mathbb T$ and a function $g\in \mathcal H_{\sigma,1},$ the classical Dirichlet space, such that
$\lim_{r\to 1^-}g(r\tilde{\lambda})=\infty$. Then by Theorem \ref{generalized local Douglas formula}, $(z-\lambda)g\in \mathcal H_{\lambda,2}$. But $\lim_{r\to 1^-}(r\tilde{\lambda}-\lambda)g(r\tilde{\lambda})=\infty$. Hence $(z-\lambda)g\notin H^\infty.$

\item[\rm (ii)] Note that $\mathcal H_{\sigma,2} = \mathcal D_2$ and $\mathcal D_2$ forms an algebra, see \cite[Theorem 7]{TaylorDA}. In fact, if $\pmb \mu=(\mu_1,\mu_2+\sigma)$, where $\mu_1,\mu_2\in \mathcal M_+(\mathbb T)$, then $\mathcal H_{\pmb\mu}$ is  an algebra. To see this, note that $\mathcal H_{\pmb\mu} = \mathcal H_{\mu_1}\cap\mathcal H_{\sigma,2} \cap \mathcal H_{\mu,2}.$ Since $\mathcal H_{\sigma,2}\subset \text{Mult\,}(\mathcal H_{\sigma,1})$, we obtain that $\mathcal H_{\pmb\mu}\subset \text{Mult\,}(\mathcal H_{\sigma,1}).$
The conclusion now follows from Theorem \ref{prop algebra}.
\end{itemize}
\end{example}

%

%

The following theorem is a generalization of \cite[Theorem 8.3.2]{Primer}. As a corollary to this theorem we obtain a stronger version of Theorem \ref{prop algebra} of this section.
\begin{theorem}\label{Inv-Mult}
Let $\mu\in \mathcal M_+(\mathbb T)\setminus \{0\},$  $n\in\mathbb N.$ Let $\mathcal W$ be a closed $M_z$-invariant subspace of $\mathcal H_{\mu,n}$ and $\phi\in \text{Mult\,}(\mathcal H_{\sigma,n-1}).$ Suppose $f\in \mathcal W$ be such that $\phi f\in \mathcal H_{\mu,n},$ then $\phi f\in \mathcal W.$
\end{theorem}
\begin{proof}
Let $\phi$ be a multiplier of $\mathcal H_{\sigma, n-1}.$ For every $0<r<1,$ define $\phi_r(z):=\phi(rz).$ 
Note that for every $0<r<1,$ $\phi_r\in \mathcal O(\overline{\mathbb D})$ and 
hence $\phi_r\in \text{Mult\,}(\mathcal H_{\mu,n}).$ Fix a $r$ satisfying $0 < r < 1.$ Choose a sequence of polynomials $\{p_j\}_{j\in\mathbb N}$ such that $D_{\mu,n}(p_j-\phi_r) \to 0$ as $j\to \infty.$ Let $f\in \mathcal W.$ Since $\mathcal W$ is a $M_z$-invariant subspace of $\mathcal H_{\mu,n},$ we have that $p_j f\in \mathcal W$ for all $j\in\mathbb N.$ As $p_nf\to \phi_r f$ in $\mathcal H_{\mu, n},$ we have that $\phi_r f\in M.$ Since $r$ is an  arbitrary number between $0$ and $1,$ we conclude that $\phi_r f\in M$ for every $0< r <1.$ To show $\phi f\in M,$ it is enough to show that 
$$\sup\{D_{\mu,n}(\phi_r f):0<r<1\}< \infty.$$ 
Since if this is the case then $\{\phi_r f\}$ turns out to be a norm bounded family of function in the Hilbert space $\mathcal H_{\mu,n}.$ And therefore there exists a sequence $r_n$ converging to $1$ such that $\phi_{r_n} f$ converges to $\phi f$ weakly and hence $\phi f\in M.$ Using Lemma \ref{Multiplier property}, we find that 
\begin{eqnarray*}
D_{\lambda,n}(\phi_r f) & \leqslant & 2 \|\phi_r\|_{s (\sigma,n-1)}^2 D_{\lambda,n}(f) + 2 |f(\lambda)|^2D_{\lambda,n}(\phi_r)\\
& \leqslant & 6\|\phi\|_{s (\sigma,n-1)}^2 D_{\lambda,n}(f)+ 4 D_{\lambda,n}(\phi f)
\end{eqnarray*}   
Integrating with respect to $\mu$ on both the sides of the above inequality, we obtain that 
$$\sup \,\{ D_{\mu,n}(\phi_r f):0<r<1\}< \infty.$$
This completes the proof.
\end{proof}

\begin{cor}
Let $\mu\in \mathcal M_+(\mathbb T)\setminus \{0\}$ and $n\in \mathbb N$. If $\mathcal W$ is a closed $M_z$-invariant subspace of $\mathcal H_{\mu,n}$ then $\mathcal W \cap \text{Mult\,}(\mathcal H_{\sigma,n-1})$ is an algebra.
\end{cor}
\begin{proof}
Let $\varphi, f\in \mathcal W \cap \text{Mult\,}(\mathcal H_{\sigma,n-1}).$ From Theorem \ref{prop algebra}, we know that $\varphi f\in \mathcal H_{\mu,n}.$ Now by Theorem \ref{Inv-Mult}, we deduce that $\varphi f\in \mathcal W.$ This completes the proof.
\end{proof}
We have already seen that $\mathcal H_{\mu,n}$ is an algebra for every $n\geqslant 3$ and for any $\mu \in \mathcal M_+(\mathbb T).$ There are many interesting consequences of a reproducing kernel Hilbert space $\mathscr H_{\kappa}$ being an algebra under pointwise product, see for example \cite[Proposition 31, pp. 94--95 ]{SH}. In this part of the section, we note down some of these consequences of $\mathscr H_{\kappa}$ being an algebra. Let $\mathscr C(\overline{\mathbb D})$ denote the algebra of continuous functions on the closed unit disc $\overline{\mathbb D}.$ In the following proposition, we first derive that under certain conditions, an abelian Banach algebra contained in $\mathcal O(\mathbb D)$ must lie in the disc algebra $\mathcal O(\mathbb D)\cap \mathscr C(\overline{\mathbb D}).$ 
\begin{prop}\label{algebra is in disc algebra}
Let $\mathcal A \subseteq \mathcal O(\mathbb D)$ be an abelian Banach algebra under pointwise product and the set of all polynomials is dense in $\mathcal A.$ Suppose $\sigma_\mathcal A(z),$ the spectrum of the co-ordinate function $z$ in $\mathcal A,$ is equal to the closed unit disc $\overline{\mathbb D}.$ Then it follows that $\mathcal A \subseteq \mathscr C(\overline{\mathbb D})$ and for every $\lambda\in\overline{\mathbb D},$ the evaluation map $ev_{\lambda}: \mathcal A  \to \mathbb C$ defined by $ev_{\lambda}(f)= f(\lambda),\,f\in\mathcal A,$ is continuous. Furthermore $\sigma_{\mathcal A}(f)= f(\overline{\mathbb D})$ holds for every $f\in\mathcal A.$
\end{prop}
\begin{proof}
Let $\Sigma$ be the maximal ideal space of $\mathcal A,$ that is, the collection of all non-zero homomorphisms of $\mathcal A$ into $\mathbb C.$ Consider $\Sigma$ with the weak* topology that it inherits as a subset of $\mathcal A^*.$ It is well known that $\Sigma$ is a compact Hausdorff space, see \cite[Ch. VII, Theorem 8.6]{Con}. Since 
$\mathcal A \subseteq \mathcal O(\mathbb D),$ it follows that $ev_{\lambda}$ belong to $\Sigma$ for every $\lambda\in\mathbb D.$  By our assumption the co-ordinate function $z$ is a generator for the algebra $\mathcal A.$ By \cite[Ch. VII, Proposition 8.10]{Con}, we have that there is a homeomorphism $\tau: \Sigma \to \overline{\mathbb D}$ such that $\tau(ev_{\lambda})= \lambda.$ Furthermore, it follows that the Gelfand transform $\gamma : \mathcal A \to \mathscr C (\overline{\mathbb D}),$ satisfies the relation $\gamma(p)= p$ for every polynomial $p$ in $\mathbb C[z].$  Now let $f\in\mathcal A.$ For each $\lambda\in \mathbb D,$ we have that 
\begin{align*}
f(\lambda)= ev_{\lambda}(f)= \big(\tau^{-1}(\lambda)\big)(f).
\end{align*}
So for each $\lambda \in \mathbb T,$ it is natural to define $f(\lambda):=\big(\tau^{-1}(\lambda)\big)(f).$ 
Since $\tau$ is a homeomorphism, it follows that $f$ extends as a function in $\mathscr C(\overline{\mathbb D}).$
Note that $ev_{\lambda}= \tau^{-1}(\lambda)\in \Sigma$ for every $\lambda\in \overline{\mathbb D}.$ 
Hence the continuity of  $ev_{\lambda},$ for each $\lambda \in \overline{\mathbb D},$ follows from \cite[Ch. VII, Proposition 8.4]{Con}. Finally, in view of \cite[Ch. VII, Theorem 8.6]{Con}, it follows that 
\begin{align*}
\sigma_{\mathcal A}(f) = \{h(f): h\in\Sigma\}= \{\tau^{-1}(\lambda)(f): \lambda\in\overline{\mathbb D}\}= \{f(\lambda) :\lambda\in\overline{\mathbb D}\}.
\end{align*}
This completes the proof.
\end{proof}

As an application of the aforementioned proposition, we obtain that every functional Hilbert space $\mathscr H_{\kappa}$ with property (A), which is also an algebra under pointwise product, must lie in the disc algebra. Furthermore, for such instances, one can describe all the cyclic vectors in $\mathscr H_{\kappa}.$ Recall that a vector $f$ in $\mathscr H_{\kappa}$ is said to be cyclic if $\mathscr H_{\kappa}$ is the smallest closed $M_z$-invariant subspace containing $f.$ The following results can be thought as a straightforward generalization of the results in \cite[Corollary 2, Corollary 3, p. 95]{SH}. 

\begin{theorem}
Let $\mathscr H_{\kappa}$ be a functional Hilbert space with property (A). Suppose $\mathscr H_{\kappa}$ is an algebra under pointwise product and the spectrum $\sigma(M_z)= \overline{\mathbb D}.$ Then $\mathscr H_{\kappa} \subseteq \mathscr C(\overline{\mathbb D})$ and every closed $M_z$-invariant subspace of $\mathscr H_{\kappa}$ is an ideal.  Furthermore, for a vector $f\in\mathscr H_{\kappa},$ the following are equivalent:
\begin{itemize}
\item[(a)] $f$ is cyclic.
\item[(b)] $f$ is invertible in $\text{Mult\,}(\mathscr H_{\kappa}).$
\item[(c)] $f$ has no zero in $\overline{\mathbb D}.$
\end{itemize}
\end{theorem}
\begin{proof}
Note that  $\mathscr H_{\kappa}$ is an algebra under pointwise product if and only if $\mathscr H_{\kappa}= \text{Mult\,}(\mathscr H_{\kappa}).$  Since $\|1\|_{\mathscr H_{\kappa}}= \kappa (0,0)=1,$ it follows that $\|f\|_{\mathscr H_{\kappa}} \leqslant \|M_f\|=\|f\|_{op}$ for every $f\in \mathscr H_{\kappa}.$ This shows that the inclusion map from $\big(\text{Mult\,}(\mathscr H_{\kappa}),\|\cdot\|_{op}\|\big)$ onto $\big(\mathscr H_{\kappa},\|\cdot\|_{\mathscr H_{\kappa}}\big)$ is continuous. By an application of open mapping theorem it follows that  the two norms $\|\cdot\|_{op}$ and $\|\cdot\|_{\mathscr H_{\kappa}}$ on $\mathscr H_{\kappa}$ are equivalent. By assumption (A3), it now follows that the set of all polynomials is dense in $\big(\text{Mult\,}(\mathscr H_{\kappa}),\|\cdot\|_{op}\|\big)$ as well. Since  $\text{Mult\,}(\mathscr H_{\kappa})$  coincides with the commutant of $M_z$ in  $\mathcal B(\mathcal H_{\kappa}),$  it follows that the spectrum of the co-ordinate function $z$ in the multiplier algebra $\text{Mult\,}(\mathscr H_{\kappa})$ is equal to the spectrum of $M_z$ in $\mathscr B(\mathscr H_{\kappa}).$ Now applying  Proposition \ref{algebra is in disc algebra} to the abelian Banach algebra $\mathcal A = \big(\text{Mult\,}(\mathscr H_{\kappa}),\|\cdot\|_{op}\big),$ we obtain that $\mathscr H_{\kappa}= \text{Mult\,}(\mathscr H_{\kappa})\subseteq \mathscr C(\overline{\mathbb D}).$
For the proof of the second part, consider an element $f\in \mathscr H_{\kappa}.$ Let $\mathscr I(f)$ denote the smallest closed $M_z$-invariant subspace of $\mathscr H_{\kappa},$ containing $f,$ that is,
\begin{align*}
\mathscr I(f)= \bigvee\{z^nf: n\in \mathbb Z_{\geqslant 0}\}.
\end{align*}
To show every closed $M_z$-invariant subspace is an ideal, it is sufficient to show that $\mathscr I(f)$ is an ideal. 
As the two norms $\|\cdot\|_{op}$ and $\|\cdot\|_{\mathscr H_{\kappa}}$ on $\mathscr H_{\kappa}$ are equivalent and the set of all polynomials is dense in $\mathscr H_{\kappa},$ it follows that $\mathscr I(f)$ is an ideal. 
Since $\mathscr I(f)$ is a closed ideal in the Banach algebra $\mathcal A = \big(\text{Mult\,}(\mathscr H_{\kappa}),\|\cdot\|_{op}\|\big)$, then either one of the following two mutually exclusive event must holds.
\begin{itemize}
\item[(i)] $\mathscr I(f)= \mathscr H_{\kappa},$ 
\item[(ii)] $\mathscr I(f)$ is contained in a maximal ideal, that is, $\mathscr I(f)\subseteq \ker (ev_{\lambda})$ for some $\lambda\in\overline{\mathbb D}.$ In particular $f(\lambda)=0$ for some $\lambda\in\overline{\mathbb D}.$ 
\end{itemize}
Hence we obtain that an element $f\in \mathscr H_{\kappa},$ is cyclic , that is, $\mathscr I(f)= \mathscr H_{\kappa}$ if and only if $f$ has no zeros in $\overline{\mathbb D}.$ Applying Proposition \ref{algebra is in disc algebra}, we also have that $\sigma_{\mathcal A}(f)= \{f(\lambda): \lambda\in \overline{\mathbb D}\}.$ Thus $f$ has no zeros in $\overline{\mathbb D}$ if and only if $f$ is invertible in $\text{Mult\,}(\mathscr H_{\kappa}).$ This completes the proof.
\end{proof}

\section{Codimension $k$ property of Invariant subspaces}\label{Cod}
This section is an attempt to study the codimension $k$ property of  non-zero closed $M_z$-invariant subspaces of $\mathcal H_{\pmb\mu}$  for any $\pmb \mu =(\mu_1,\ldots,\mu_{m})$ with $\mu_j\in\mathcal M_+(\mathbb T)$ for $j=1,\ldots,m.$  
Since $M_z|_{_\mathcal M}$ is an analytic, $(m+1)$-isometry for every non-zero $\mathcal M$ in $Lat(M_z,\mathcal H_{\pmb\mu}),$ it turns out that the dimension of  $\ker \big ((M_z|_{_\mathcal M})^*-\bar{\lambda}\big),$ that is, the dimension of $\mathcal M\ominus (z-\lambda)\mathcal M,$ is a positive integer independent of $\lambda \in\mathbb D,$ see \cite[Proposition 6.4]{GGR}. Thus any $\mathcal M$ in $Lat(M_z,\mathcal H_{\pmb \mu})$ has codimension $k$ property if and only if $ dim (\mathcal M \ominus z\mathcal M)=k.$ In case of $m=1,$ that is, $\pmb\mu=\mu_1$ and $\mu_1\in \mathcal M_{+}(\mathbb T),$ it is well known that for any non-zero $\mathcal M$ in $Lat(M_z,\mathcal H_{\pmb\mu}),$ $\mathcal M$ has codimension $1$ property, see \cite[Theorem 3.2]{RS-92}. Here, firstly we show that for any $m\geqslant 3$ and any non-zero $\mathcal M$ in $Lat(M_z,\mathcal H_{\pmb\mu})$ has codimension $1$ property.


\begin{theorem}\label{algebra cod}
Let $\pmb \mu =(\mu_1,\ldots,\mu_{m})$ be an $m$-tuple of finite non-negative Borel measure on $\mathbb T.$ Suppose $\mathcal H_{\pmb\mu}$ is an algebra under pointwise product, that is, $\mathcal H_{\pmb\mu}= Mult(\mathcal H_{\pmb\mu}).$ Then for every 
non-zero $\mathcal M$ in $Lat(M_z,\mathcal H_{\pmb \mu}),$ we have $\dim (\mathcal M\ominus z\mathcal M) = 1.$ In particular if $m\geqslant 3$ and $\mu_m\neq 0,$ then $\dim (\mathcal M\ominus z\mathcal M) = 1$ for any $\mathcal M$ in $Lat(M_z,\mathcal H_{\pmb\mu}).$
\end{theorem}
\begin{proof}
Suppose $\mathcal H_{\pmb\mu}$ is an algebra under pointwise product, that is, $\mathcal H_{\pmb\mu}= Mult(\mathcal H_{\pmb\mu}).$ By an application of Closed Graph Theorem  the multipliers norm $\|\cdot\|_{op}$ and the Hilbertian norm $\|\cdot\|_{\pmb\mu}$ on  $\mathcal H_{\pmb\mu}$ are equivalent.  In view of  \cite[Corollary 5.3]{GGR}, we obtain that the set of all polynomials is dense in $\mathcal H_{\pmb\mu}$ w.r.t the multiplier norm $\|\cdot\|_{op}.$ For any $f\in\mathcal H_{\pmb \mu},$ let $[f]$ denotes the smallest closed $M_z$-invariant subspace containing $f$ in $\mathcal H_{\pmb\mu}.$ Hence it follows that $\varphi [f] \subseteq [f]$ for every $\varphi,f \in \mathcal H_{\pmb \mu}.$ Applying \cite[Proposition 1]{Bourdon} to $\mathcal H_{\pmb\mu}$, we get that $\mathcal H_{\pmb\mu}$ is a cellular-indecomposable space. Now the first part of the theorem follows immediately from \cite[Theorem 1]{Bourdon}. The last part follows using the first part together with Corollary \ref{algebra for tuple}.
\end{proof}
For the remaining case $m=2,$ that is, $\pmb \mu =(\mu_1,\mu_{2})$ with $\mu_1,\mu_2\in\mathcal M_+(\mathbb T),$ we answer this question partially in Theorem \ref{case m=2}. We start with a basic lemma which will be useful in the proof of Theorem \ref{case m=2}.
\begin{lemma}\label{Rank-Nullity}
Let $H$ be a Hilbert space and $V$ be a closed subspace of $H$ with $\dim V^\perp=n.$ Let $W$ be another subspace of $H$, then $\dim (W\ominus (V\cap W))\leqslant n.$
\end{lemma}
\begin{proof}
Let $P$ denote the orthogonal projection of $H$ onto $V^\perp$. Suppose $X=W\ominus (V\cap W).$ We claim that $P_{|_X}$ is an injective operator. To verify the claim, let $x\in X$ be such that $P_{|_X}(x)=0$. Since $\ker P=V,$ $x\in V.$ This implies $x=0.$ Thus claims stands verified and therefore $\dim (W\ominus (V\cap W))\leqslant \dim V^\perp =n$.
\end{proof}


\begin{lemma}\label{restriction}
Let $\mu$ be a finite positive Borel measure on the unit circle $\mathbb T$. Let $n\in\mathbb N,$ $\lambda_1,\ldots, \lambda_n\in \mathbb T$ be distinct points and $c_1,\ldots, c_n$ be positive real numbers. Let $\pmb\mu$ be the tuple $(\mu, \sum_{j=1}^{n}c_j \delta_{\lambda_j})$.
Then the subspace $\ker( {M_z^*}^2M_z^2 - 2 M_z^* M_z+I) = \mathcal V$ (say) is a closed $M_z$-invariant subspace of $\mathcal H_{\boldsymbol \mu}$ and the operator $M_z|_{\mathcal V}$ is a cyclic analytic $2$-isometry. Moreover, $\dim \mathcal V^\perp = n.$
\end{lemma}
\begin{proof}
From \cite[Theorem 4.1]{GGR}, we know that $M_z$ on $\mathcal H_{\pmb\mu}$ is a cyclic analytic $3$-isometry. Thus, by \cite[Proposition 1.6]{AglerStan1}, it follows that $\mathcal V$ is a closed $M_z$-invariant subspace of $\mathcal H_{\pmb\mu}$ and the operator $M_z|_{\mathcal V}$ is a  $2$-isometry. Since $M_z$ on $\mathcal H_{\pmb\mu}$ is analytic, $M_z|_{\mathcal V}$ is also analytic. In view of \cite[Theorem 1]{Rinv}, to show $M_z|_{\mathcal V}$ is cyclic, it is sufficient to show that $dim(\mathcal V\ominus z\mathcal V)=1.$

Before proving that $dim(\mathcal V\ominus z\mathcal V)=1,$
first we show that $\dim \mathcal V^\perp = n$. In this regard,  note that $\mathcal H_{\pmb \mu}\subseteq \mathcal H_{\lambda_j, 2}$ for each $j=1,\ldots,n.$ Thus, by Theorem \ref{generalized local Douglas formula}, $f^*(\lambda_j)$ exists for all $f\in\mathcal H_{\pmb \mu}$ and for every $j=1,\ldots,n$. Since $\mathcal H_{\pmb \mu}$ is a reproducing kernel Hilbert space on $\mathbb D,$ the linear functional $f \mapsto f(r\lambda_j)$ is bounded on $\mathcal H_{\pmb \mu}$ for any $0 < r <1 $ and for each $j=1,\ldots,n.$ An application of uniform boundedness principle gives that $f\mapsto f^*(\lambda_j)$ is a bounded linear functional on $\mathcal H_{\pmb \mu}$ for each $j=1,\ldots,n$.
Therefore, by Riesz representation theorem, there exist $h_1,\ldots,h_n\in \mathcal H_{\pmb \mu}$ such that 
\begin{equation}\label{eqn16}
f^*(\lambda_j)=\langle f, h_j\rangle, ~~f\in \mathcal H_{\pmb\mu},\,\,\, j=1,\ldots,n.
\end{equation}
Since $M_z$ on $\mathcal H_{\pmb\mu}$ is a $3$-isometry, by \cite[Proposition 1.5]{AglerStan1}, the operator 
${M_z^*}^2M_z^2 - 2 M_z^* M_z+I$ is positive. Thus
a function $f$ of $\mathcal H_{\pmb \mu}$ is in $\mathcal V$ if and only if 
$\|z^2f\|_{\pmb \mu}^2-2\|zf\|_{\pmb \mu}^2+\|f\|_{\pmb \mu}^2=0.$ By  \cite[Proposition 2.7]{GGR}, this is equivalent to $D_{\tau, 0}(f)=0,$ where $\tau=\sum_{j=1}^{n}c_j \delta_{\lambda_j}.$ Since $D_{\tau, 0}(f)=\sum_{j=1}^n c_j|f^*(\lambda_j)|^2$, we see that 
$$\mathcal{V}=\{f\in\mathcal H_{\boldsymbol \mu}: f^*(\lambda_j)=0, 1\leqslant j\leqslant n\}.$$
Hence, by \eqref{eqn16}, it follows that $\mathcal V^\perp=\mbox{span}~\{h_1,\ldots,h_n\}$. Since $\lambda_j'$s are distinct points on $\mathbb T$, it follows that 
$h_1,\ldots,h_n$ are linearly independent. This shows that $\dim \mathcal V^\perp =n$. 
To finish the proof, note that $M_z$ on $\mathcal H_{\pmb\mu}$ is in the Cowen-Douglas class $B_1(\mathbb D)$, see \cite[Corollary 6.1]{GGR}.  Since $\mathcal V$ is a closed $M_z$-invariant subspace of finite co-dimension, it follows from \cite[p. 71]{CDM} that $\dim (\mathcal V\ominus z\mathcal V) = 1.$ This completes the proof.
\end{proof}

\begin{rem}
By Remark \ref{remalgebra}\rm (ii), the space $\mathcal H_{\pmb\mu}$ is an algebra, where $\pmb \mu=(\mu_1,\mu_2+\sigma)$, $\mu_1,\mu_2\in \mathcal M_+(\mathbb T).$
Hence in view of Theorem \ref{algebra cod} we obtain that  for any $\mathcal M\in Lat(M_z,\mathcal H_{\pmb\mu}),$  $\dim (\mathcal M\ominus z\mathcal M) = 1.$ 
\end{rem}

\begin{theorem}\label{case m=2}
Let $\mu$ be a finite non-negative Borel measure on the unit circle $\mathbb T$. Let $n\in\mathbb N,$ $\lambda_1,\ldots, \lambda_n\in \mathbb T$ be distinct points and $c_1,\ldots, c_n$ be positive real numbers. Suppose $\mathcal M$ is a non-zero subspace in $Lat(M_z,\mathcal H_{\pmb\mu}),$ where $\pmb\mu=(\mu, \sum_{j=1}^{n}c_j \delta_{\lambda_j})$. Then $\dim (\mathcal M\ominus z\mathcal M) = 1.$ 
\end{theorem}
\begin{proof}
Let $\mathcal V$ be the subspace
$\ker( {M_z^*}^2M_z^2 - 2 M_z^* M_z+I)$. By Lemma \ref{restriction}, the  operator $T:=M_z|_{_{\mathcal V}}$ is a cyclic analytic $2$-isometry. As $\mathcal V\cap \mathcal M=\mathcal N$ (say) is an invariant subspace for $T$, from \cite[Theorem 3.2]{RS-92}, it follows that 
\begin{eqnarray}\label{Reduction on 2-isometry}
\dim (\mathcal N\ominus z\mathcal N)\leqslant 1.
\end{eqnarray}
 Let $\mathcal L$ be a subspace of $\mathcal M$ such that $\mathcal M=\mathcal N\oplus \mathcal L.$  Then from Lemma \ref{Rank-Nullity}, it follows that $\dim \mathcal L\leqslant n.$ 
Since $M_z$ is an injective map, $z\mathcal N\cap z\mathcal L=\{0\}$ and $\dim \mathcal L=\dim z\mathcal L$ as $\mathcal L$ is finite dimensional. 
From \eqref{Reduction on 2-isometry}, we now have
\begin{eqnarray*}
\dim (\mathcal M\ominus z\mathcal M) &=& \dim ((\mathcal N\oplus \mathcal L) \ominus z(\mathcal N\oplus \mathcal L))\\
&=& \dim ((\mathcal N\oplus \mathcal L) \ominus (z\mathcal N + z\mathcal L))\\
&\leqslant & 1.
\end{eqnarray*} 
This completes the proof of the theorem.
\end{proof}
\section{Relationship between $\mathcal H_{\pmb \mu}$ and $\mathcal H(b)$}\label{Relationship with H(b)}
%
In what follows, for two reproducing kernel Hilbert spaces $\mathcal H_1$ and $\mathcal H_2$ consisting of holomorphic functions on the unit disc $\mathbb D$, we write $\mathcal H_1=\mathcal H_2$ to mean that $\mathcal H_1$ and $\mathcal H_2$ are equal as a set but the norms are not necessarily same. In case $\mathcal H_1=\mathcal H_2,$ using closed graph theorem, it can be shown that the norms on $\mathcal H_1$ and $\mathcal H_2$ must be equivalent. 


For a positive measure $\mu$ in $\mathcal M_+(\mathbb T)$, we set the notation (as in \cite{CostaraRansford})
\begin{equation*}
V_{\mu}(z):=\int_{\mathbb T}\frac{d\mu(\lambda)}{|1-\lambda\overline{z}|^2},\quad z\in \mathbb D.
\end{equation*}
With the help of generalized local Douglas formula, for a fixed $w\in\mathbb D,$ the following lemma computes the semi-norm of Szeg\"o kernel $S(z,w):=\frac{1}{1-z\overline{w}}$ in $\mathcal H_{\mu, n},$ in terms of $V_{\mu}(w).$ 
\begin{lemma}\label{norm of hardy kernel}
Let $\mu$ be a finite positive Borel measure on $\mathbb T$ and $n$ be a positive integer.
Then 
\begin{equation*}
D_{\mu,n}\Big(\frac{1}{1-z\overline{w}}\Big)=\frac{|w|^{2n}}{(1-|w|^2)^n}V_{\mu}(w),\quad w\in \mathbb D.
\end{equation*}
\end{lemma}
\begin{proof}
Let $\lambda\in \mathbb T$. Then by Theorem \ref{generalized local Douglas formula}, we have
\begin{align*}
D_{\lambda,n}\Big(\frac{1}{1-z\overline{w}}\Big)&=D_{\sigma,n-1}\Big(\frac{\frac{1}{1-z\overline{w}}-\frac{1}{1-\lambda\overline{w}}}{z-\lambda}\Big)\\
&=D_{\sigma,n-1}\Big(\frac{\overline{w}}{(1-z\overline{w})(1-\lambda\overline{w})}\Big)\\
&=\frac{|w|^2}{|1-\lambda\overline{w}|^2}D_{\sigma,n-1}\Big(\frac{1}{1-z\overline{w}}\Big).
\end{align*}
To complete the proof for the case $\mu$ is point mass measure $\delta_\lambda$ at the point $\lambda,$
we claim that 
\begin{eqnarray}\label{D-sigma(Szego)}
D_{\sigma,n-1}\Big(\frac{1}{1-z\overline{w}}\Big)=\frac{|w|^{2n-2}}{(1-|w|^2)^n}.
\end{eqnarray}
Since $\frac{1}{1-z\overline{w}}$ is the reproducing kernel for the Hardy space $H^2,$ it is easy to verify \eqref{D-sigma(Szego)} for the case $n=1.$ Suppose $n>1$, then by definition of $D_{\sigma,n-1}(\cdot)$ and the fact that $\frac{1}{(1-z\overline{w})^{n}}$ is the reproducing kernel of the weighted Bergman space $A^2\big((n-1)(1-|z|^2)^{n-2}dA(z)\big)$, we get
\begin{align*}
D_{\sigma,n-1}\Big(\frac{1}{1-z\overline{w}}\Big)&=(n-1)|w|^{2n-2}\int_{\mathbb D}\frac{(1-|z|^2)^{n-2}}{|1-z\overline{w}|^{2n}} dA(z)\\
&=\frac{|w|^{2n-2}}{(1-|w|^2)^n}.
\end{align*}
Hence equation\eqref{D-sigma(Szego)} stands verified and consequently we have
$$D_{\lambda,n}\Big(\frac{1}{1-z\overline{w}}\Big)=\frac{|w|^{2n}}{|1-\lambda\overline{w}|^2(1-|w|^2)^n}.$$
Integrating both the sides with respect to $\mu,$ we get
$$D_{\mu,n}\Big(\frac{1}{1-z\overline{w}}\Big)=\int_{\mathbb T}D_{\lambda,n}\Big(\frac{1}{1-z\overline{w}}\Big)d\mu(\lambda)=\frac{|w|^{2n}}{(1-|w|^2)^n}V_{\mu}(w).$$
This completes the proof of the lemma.
\end{proof}


In the following main result of this section,
we prove that the spaces $\mathcal H_{\pmb \mu}$ ($m\geqslant 2)$ and  $\mathcal H(b)$ can not be same even as sets. 

\begin{theorem}\label{never de Branges-Rovnyak}
Let $m\geqslant 2$ and $\pmb \mu= (\mu_1,\ldots,\mu_{m})$ be an $m$-tuple of positive measures in $\mathcal M_{+}(\mathbb T)$. If  $ \mathcal H_{\pmb \mu}=\mathcal H(b)$ for some $b\in H^\infty$ with $\|b\|_{\infty}\leqslant 1$ then 
$$\mu_2=\cdots=\mu_{m}=0.$$ 
\end{theorem}
\begin{proof}
Suppose that there exists a $b\in H^{\infty}$  with $\|b\|_{\infty}\leqslant 1$ such that $\mathcal H(b)= \mathcal H_{\pmb \mu}$. Then, by an application of closed graph theorem, there exist constants $\alpha>0$ and $\beta>0$ such that 
\begin{equation}\label{equiv norms}
\beta \|f\|_{\mathcal H(b)}\leqslant\|f\|_{\pmb \mu}\leqslant \alpha\|f\|_{\mathcal H(b)},~f\in \mathcal H(b).
\end{equation}
Since $ \mathcal H_{\pmb \mu}$ contains all functions holomorphic on a neighbourhood of $\overline{\mathbb D}$, it follows from \cite[Theorem 2.1]{CostaraRansford} that $b$ is not an extreme point of the unit ball of $H^\infty$. Hence there exists a outer function $a$ such that $(b,a)$ is a pair. Note that the function $\frac{1}{1-z\overline{w}}$ is holomorphic on a neighbourhood of $\overline{\mathbb D}$ for all $w\in \mathbb D$. By  \cite[Lemma 3.4]{CostaraRansford}, we have 
\begin{equation*}
\Big\|\frac{1}{1-z\overline{w}}\Big\|^2_{\mathcal H(b)}=\frac{1+\big|\frac{b(w)}{a(w)}\big|^2}{1-|w|^2},~w\in \mathbb D.
\end{equation*}    
Also by Lemma \ref{norm of hardy kernel} we have 
\begin{equation*}
\Big\|\frac{1}{1-z\overline{w}}\Big\|^2_{\pmb \mu}=\frac{1}{1-|w|^2}+\sum_{j=1}^{m-1}\frac{|w|^{2j}}{(1-|w|^2)^j}V_{\mu_j}(w),~w\in \mathbb D.
\end{equation*}
Hence by \eqref{equiv norms} we obtain
\begin{equation*}
\frac{1}{1-|w|^2}+\sum_{j=1}^{m-1}\frac{|w|^{2j}}{(1-|w|^2)^j}V_{\mu_j}(w)\leqslant \alpha~ \frac{1+\big|\frac{b(w)}{a(w)}\big|^2}{1-|w|^2},~~w\in \mathbb D.
\end{equation*}
Thus for $2\leqslant j\leqslant m-1$, we have
$$\frac{|w|^{2j}}{(1-|w|^2)^j}V_{\mu_j}(w)\leqslant  \alpha~ \frac{1+\big|\frac{b(w)}{a(w)}\big|^2}{1-|w|^2}.$$
Since  $V_{\mu_j}(w)\geqslant \frac{\mu_j(\mathbb T)}{4}$ for all $w\in \mathbb D$, we get that 
\begin{equation*}
\mu_j(\mathbb T)\leqslant 4\alpha (1-|w|^2)^{j-1}\frac{1+\big|\frac{b(w)}{a(w)}\big|^2}{|w|^{2j}}.
\end{equation*}
Let $\zeta$ 
be any point in $\mathbb T$ such that $\frac{b^*(\zeta)}{a^*(\zeta)}$ exists. Since $j\geqslant 2$, putting $w=r\zeta$ in the above inequality and taking limit $r\to 1^-$, we get that $\mu_j(\mathbb T)=0$, completing the proof of the theorem. 
\end{proof}



 \begin{rem}
Let  $\boldsymbol{B}:\mathbb D\to (\ell^2)_1$ be an analytic map with $\boldsymbol B(z)=(b_i(z))_{i=0}^\infty$, where $(\ell^2)_1$ is the closed unit ball in $\ell^2$. The space  $\mathcal H(\boldsymbol B)$ corresponding to $\boldsymbol B$ is defined to be the reproducing kernel Hilbert space with the kernel function $\frac{1-\sum_{i=0}^\infty b_i(z)\overline{b_i(w)}}{1-z\overline w}$, $z,w\in \mathbb D$. The space $\mathcal H(\boldsymbol B)$  
is said to be of finite rank if there exists a $N\in\mathbb N$ and an analytic map $\boldsymbol{C}:\mathbb D\to (\ell^2)_1$,  $\boldsymbol C(z)=(c_i(z))_{i=0}^\infty$ such that $c_i=0$ for all $i> N$   
and $\mathcal H(\boldsymbol B)=\mathcal H(\boldsymbol C),$ with equality of norms. 
 Since the backward shift operator  $L$, defined by $(Lf)(z)=\frac{f(z)-f(0)}{z}$, $z\in\mathbb D$, is contractive on $\mathcal H_{\pmb\mu}$ (see \cite[Lemma 2.9]{GGR}), it follows from \cite[Proposition 2.1]{AM} that $\mathcal H_{\pmb\mu}$ coincides with the space $\mathcal H(\boldsymbol B)$ for some $\boldsymbol B$ with $\boldsymbol B(0)=0$. 
 In the case of finite rank $\mathcal H(\boldsymbol B)$ spaces with $\boldsymbol B(0)=0$, if the operator $M_z$ is bounded, then it turns out that  $\overline{\mbox{ran}}(M_z^*M_z-I)$ is finite dimensional, (see \cite[Lemma 5.1]{LGR}).
On the contrary,  for the space $\mathcal H_{\pmb\mu}$, where $\pmb\mu=(\mu_1,\ldots,\mu_{m})$, $m\geq 2$, it is worth noting that the space $\ker(M_z^*M_z-I)$ is finite dimensional. This follows easily from \cite[Proposition 2.7]{GGR}. Therefore, in this case, $\overline{\mbox{ran}}(M_z^*M_z-I)$ is infinite dimensional.
Consequently, we conclude that if $m\geqslant 2$, then  for any $\pmb\mu=(\mu_1,\ldots,\mu_{m})$, $\mu_{m}\neq 0$,
$\mathcal H_{\pmb \mu}$ can not be equal to any $\mathcal H(\boldsymbol B)$ space of finite rank with equality of norms. It will be interesting to know whether for a fixed $\pmb\mu=(\mu_1,\ldots,\mu_{m}),$ does there exists a  $\mathcal H(\boldsymbol B)$ space of finite rank such that $\mathcal H_{\pmb \mu}=\mathcal H(\boldsymbol B)$ with just equivalence of norms.
\end{rem}

\begin{rem}\label{Differs}
The last part of \cite[Theorem 1.1]{LGR} says that if $M_z$ on $\mathcal H(b)$ is a strict $2m$-isometry for some $m\in\mathbb N$ then there exists a $\zeta\in\mathbb T$ such that $\mathcal H(b)$ is equal to $\mathcal D_\zeta^m$-space defined in \cite[page 3]{LGR} with equivalence of norms. Theorem \ref{never de Branges-Rovnyak} indicates that the local Dirichlet space of order $m$ in the present article differs from that in \cite{LGR}. 
\end{rem}


\begin{thebibliography}{33}
%
\bibitem{AglerStan1}
J.~ Agler and M.~ Stankus, \emph{$m$-isometric transformations of Hilbert space, I,} Integral Equations Operator Theory, \textbf{21} (4) (1995), 383--429.

\bibitem{AglerStan2}
J.~ Agler and M.~ Stankus, \emph{$m$-isometric transformations of Hilbert space, II,} Integral Equations Operator Theory, \textbf{23} (1) (1995), 1--48.

\bibitem{AglerStan3}
J.~ Agler and M.~ Stankus, \emph{$m$-isometric transformations of Hilbert space, III,} Integral Equations Operator Theory, \textbf{24} (4) (1996), 379--421.

\bibitem{AM}
A.~ Aleman, and B.~ Malman, \emph{Hilbert spaces of analytic functions with a contractive backward shift,}
Journal of Functional Analysis, \textbf{277} (1) (2019), 157--199. 

%
%
%
\bibitem{ARO}
N.~Aronszajn,  \emph{Theory of reproducing kernels,} Transactions of the American mathematical society \textbf{68} (3) (1950), 337--404.

%
%

\bibitem{BeurlingA}
A.~ Beurling,  \emph{ Exceptional sets,} Acta mathematica, \textbf{72} (1940), 1--13.

\bibitem{Bourdon}
P.~Bourdon, \emph{Cellular-indecomposable operators and {B}eurling's theorem}, Michigan Math. J. \textbf{33(2)} (1986), 187--193.

\bibitem{deBR}
L.~De Branges and J.~ Rovnyak, \emph{Square summable power series,} Courier Corporation, 2015.

\bibitem{Chartrand}
R.~Chartrand, \emph{Multipliers and Carleson measures for $D(\mu) $}, Integral Equations and Operator Theory 45.3 (2003): 309--318.

\bibitem{CGR-2010}
N.~Chevrot, D.~Guillot and T. Ransford, \emph{De {B}ranges-{R}ovnyak spaces are {D}irichlet spaces,} J. Funct. Anal. \textbf{9} (259) (2010), 2366--2383.


%
\bibitem{Con}
J.~B.~Conway, \emph{A course in functional analysis}, \textbf{96}. Springer, 2019.  
  
\bibitem{CostaraRansford}
 C. ~Costara and T. Ransford,\emph{Which de {B}ranges-{R}ovnyak spaces are {D}irichlet spaces
 (and vice versa)?}, J. Funct. Anal. \textbf{265}(12), (2013), 3204--3218.
 
\bibitem{Douglas}
J.~Douglas, \emph{Solution of the problem of Plateau,} Trans. Amer. Math. Soc. (1931), 263--321. 

%
%
%
\bibitem{Primer}
O.~El-Fallah, K.~ Kellay, J.~ Mashreghi` and Thomas Ransford, \emph{A primer on the Dirichlet space,} {\bf{ 203}}, (2014),~ Cambridge University Press.

\bibitem{Rangroup}
O.~El-Fallah, K.~ Kellay, H. Klaja, J.~ Mashreghi, and T. ~Ransford, \emph{Dirichlet spaces with superharmonic weights and de Branges–Rovnyak spaces}, Complex Analysis and Operator Theory, \textbf{10}(1), (2016),~ 97--107.
%

\bibitem{GGR}
S.~ Ghara, R.~ Gupta, Md.~ R.~Reza, \emph{Analytic m-isometries and weighted Dirichlet-type spaces}, to appear in Journal of Operator Theory.
%
%
%
%
%
%
%
%
%
%
%
%
%
%
%
%
%
%
%

%
%
%

\bibitem{LGR}
S.~Luo, and C.~Gu, and S.~Richter, \emph{Higher order local {D}irichlet integrals and de {B}ranges--{R}ovnyak spaces,}  Adv. Math. \textbf{385}, (2021), Paper No. 107748, 47.

\bibitem{MasRans}
J.~Mashreghi and T. Ransford, \emph{Hadamard multipliers on weighted Dirichlet spaces}, Integral Equations and Operator Theory, \textbf{91}(6), (2019), 1--13.

\bibitem{RS}
S.~ Richter and C.~ Sundberg, \emph{A formula for the local Dirichlet integral,}  Michigan Math. J \textbf{38} (3), (1991), 355--379.

\bibitem{RS-92}
S.~ Richter and C.~ Sundberg, \emph{Multipliers and invariant subspaces in the Dirichlet space,}  Journal of Operator Theory \textbf{28} (1), (1992), 167--186.

%
%
%


%
%
%
%








\bibitem{CFS}
G. Chacón, E. Fricain,  and M. Shabankhah, \emph{Carleson measures and reproducing kernel thesis in Dirichlet-type spaces},
Algebra i Analiz \textbf{24} (6) (2012), 0234-0852. 




  
  
  
  
  \bibitem{CDM}
M.~J. Cowen and R.~G. Douglas, \emph{On moduli for invariant subspaces}, Invariant subspaces and other topics (Timisoara and Herculane, 1981), Operator Theory: Advances and Applications \textbf{6} (1982), pp. 65--73, Birkh\"{a}user, Basel-Boston, Mass.
  

  


 
  
\bibitem{HALP}
P.~R.~Halmos, \emph{A Hilbert space problem book,} Vol. 19. Springer Science \& Business Media, 2012.



   
 
%
%
%

  

  
	   



   
\bibitem{PAULRKHS}
V.~Paulsen and M.~Raghupathi,\emph{ An introduction to the theory of reproducing kernel Hilbert spaces,} Vol. 152. Cambridge University Press, 2016.   
   
\bibitem{Rinvbanach}
S.~Richter, \emph{Invariant subspaces in Banach spaces of analytic functions}, Transactions of the American Mathematical Society, {\bf{304.2}} (1987), 585-616.   
   
\bibitem{Rinv}
S.~ Richter, \emph {Invariant subspaces of the Dirichlet shift,} J. reine angew. Math \textbf{386} (1988), 205--220.

 \bibitem{R}
S.~ Richter, \emph {A representation theorem for cyclic analytic two-isometries,} Transactions of the American Mathematical Society {\bf 328} (1991), 325--349.






\bibitem{Rydhe}
E.~ Rydhe, \emph{Cyclic $m$-isometries and Dirichlet type spaces,} Journal of the London Mathematical Society, \textbf{99} (3), (2019), 733--756.

\bibitem{Sara}
D.~ Sarason, \emph{ Local Dirichlet spaces as de Branges-Rovnyak spaces,} Proceedings of the American Mathematical Society, \textbf{125} (7), (1997), 2133--2139.

\bibitem{SH}
A.~L.~Shields, \emph{Weighted shift operators and analytic function theory,}
Topics in operator theory, Amer. Math. Soc., Providence, R.I., Math. Surveys, \textbf{13} (1974), 49--128. 

 
 \bibitem{Stegenga}
D.~A.~ Stegenga, \emph{Multipliers of the Dirichlet space},
 Illinois J. Math., {\bf 24}(1), (1980), 113--139.
 
 \bibitem{TaylorDA}
G.~ D.~ Taylor, \emph{Multipliers on D$_\alpha$,} Transactions of the American Mathematical Society (1966), 229--240. 
\end{thebibliography}
\end{document}